\documentclass[11pt]{article}

\usepackage[UKenglish]{babel}
\usepackage[T1]{fontenc}
\usepackage[utf8]{inputenc}
\usepackage{amsmath,amssymb,amsfonts,amsthm, mathtools, mathrsfs}
\usepackage{parskip, color}
\usepackage{fancyhdr}
\usepackage{dsfont}
\usepackage{graphicx, pgf}
\usepackage{caption}
\usepackage{comment}
\usepackage{float}
\usepackage{subcaption}
\usepackage{enumerate}
\usepackage{hyperref}
\usepackage{cleveref}
\usepackage{comment}
\usepackage{todonotes}
\usepackage[left=2.5cm, right=2.5cm]{geometry}

\renewcommand{\d}{\text{\,d}}
\def\to{\rightarrow}

\newtheorem{proposition}{Proposition}[section]
\newtheorem{lemma}[proposition]{Lemma}

\newtheorem*{remark*}{Remark}
\newtheorem{theorem}[proposition]{Theorem}
\newtheorem{corollary}[proposition]{Corollary}

\author{
	Barbara Niethammer\thanks{Universität Bonn, Germany, \texttt{niethammer@iam.uni-bonn.de}}
	\hspace*{1 cm} 
	Lorena Pohl\thanks{Universität Bonn, Germany, \texttt{pohl@iam.uni-bonn.de}} 
	\hspace*{1 cm}
	Juan J. L. Velázquez\thanks{Universität Bonn, Germany, \texttt{velazquez@iam.uni-bonn.de}}
}
\title{Long-time behaviour of a nonlocal model for electroporation}
\date{\today}

\begin{document}
	\maketitle
	
	{\small \textbf{Abstract}
		In this paper we analyze a model for electroporation, a biological process in which a cell membrane exposed to an external voltage becomes permeable due to the formation and growth of nanoscale membrane pores. We prove a local stability result for asymptotic  self-similar solutions with a power-law tail. Our method relies on the  analysis of an equation for the first moment as well as comparison of solutions of the full problem to solutions of a corresponding transport problem. In particular this shows that the transport term  drives the long-time behaviour.

		\textbf{Acknowledgments}
		The authors gratefully acknowledge the financial support of the Deutsche Forschungsgemeinschaft (DFG, German Research Foundation) through the collaborative research
		centre "Analysis of criticality: from complex phenomena to models and estimates" (CRC 1720, Project-ID
		539309657) and the Bonn International Graduate School of Mathematics at the Hausdorff Center for Mathematics (EXC 2047/2, Project-ID 390685813).
	}
	
	\section{Introduction}
	We consider the problem
	\begin{equation}\label{eqn-main-parabolic}
		\left\{ 
		\begin{split}
			\partial_t f &= \partial_x \Big( \partial_x f + \Big(1 - \frac{\beta}{1 +  n_f^2 } x\Big)f   \Big), \quad x > 0 \\
			n_f &= n_f(t) = \int_0^{\infty} x f(x,t)\d x,\\
			f(0, t) &= \mu,\\
			f(x, 0) &= f_0,
		\end{split}
		\right.
	\end{equation}
	where $\beta$ and $\mu$ are positive parameters.
	 This equation is motivated by models for electroporation, a widely applied biological technique to make cellular membranes permeable to DNA or other molecules. This is achieved by exposing the cells to electric pulses, which increases the transmembrane potential, thereby influencing the formation and growth of nanoscale hydrophilic membrane pores \cite{ chen_membrane_2006, cunill-semanat_spontaneous_2019}. In the Subsection  \ref{Ss.derivation} below we present a brief derivation of the model before we describe our results on the long-time behaviour of solutions to \eqref{eqn-main-parabolic} in Section \ref{Ss.results}.
	
	\subsection{Derivation of the model}\label{Ss.derivation}
	Most existing continuum models for electroporation assume that pores and pre-pore membrane defects can be described by a single variable, namely their radius $r$. We describe by $n \d r$ the number of pores or membrane defects per unit of area with radius between $r$ and $r + \d r$ and assume that the evolution of the pores is independent, except for their collective effect in modifying the transmembrane potential as seen below. The growth of an individual pore can be modelled by the Langevin equation
	\begin{equation}\label{rgrowth}
		\d r =-   \frac{\d W}{\d r} \d r + \sqrt{2D} \d B(t),
	\end{equation}
	where $B$ denotes Brownian motion and $D$ a diffusion parameter. The energy $W$ is given by 
	$$W(r, V_m) = W_{pore}(r) - \tilde{C}_m r^2 V_m^2(t).$$ 
	Here $W_{pore}$ is the pore energy, 
	$V_m(t)$  the transmembrane potential induced by an external electrical field and $\tilde C_m$ a coefficient that will be explained below.
	Then the density $n(r,t)$ evolves according to  the Smoluchowski equation
	\begin{equation}\label{eqn-electroporation-n}
		\partial_t n = D \partial_r \Big( \partial_r n + \frac{1}{k_B T} \partial_r W n \Big) + S(r).
	\end{equation}
	Spontaneous pore formation and decay are captured by the source term
	$$S(r) = s(r)\Big(a_0e^{-\frac{W_{pore}}{k_B T}} - n\Big),$$
	where $s$ is a function supported on $[0, r^*]$ which means  that only small defects are spontaneously formed and closed. The constant $a_0 > 0$ describes the rate of defect formation and determines the number of pores at equilibrium when $V_m = 0$, in which case the system has detailed balance. This class of models was introduced in \cite{weaver_theory_1996} and forms the basis for many electroporation models considered in the literature. 
	
	We give here only a qualitative description of $W_{pore}$, which is well-agreed upon in the literature \cite{glaser_reversible_1988, debruin_modeling_1999, akimov_pore_2017, akimov_pore_2017-1} and sketched in Figure 1: we distinguish between hydrophobic pre-pore membrane defects and hydrophilic pores, the latter of which become energetically favorable at a critical radius $r^*$. We assume $W_{pore}(0) = \partial_r W_{pore}(0) = 0$, and that for small $r$ the energy $W_{pore}$ is increasing up to an energy barrier of height $E^*$ at $r^*$. After the energy barrier there is 
	a local minimum $E_0$ at $r_0 > r^*$ due to the repulsion of lipid heads. The energy of a hydrophilic pore then scales linearly like $2\pi \sigma_l r$, where $\sigma_l$ represents the line tension. For very large pores it holds
	$W_{pore}(r) \sim 2\pi \sigma_l r - \pi \sigma_s r^2,$
	where $\sigma_s$ is the surface tension. Hence there exists a second critical radius depending on the relation between $\sigma_l$ and $\sigma_s$ and corresponding to a local maximum in $W_{pore}$, after which the pore energy is quadratically decreasing. A pore that passes this critical radius will continue to grow, ultimately causing membrane breakdown. However, this second energy barrier must be very high, since otherwise some pores could cross it due to fluctuations  in times of order one, leading to membrane breakdown even when no voltage is applied. Since we want to model a non-breakdown case, we disregard this quadratic term in the following, that is we assume that $\sigma_s=0$.
	\bigskip
	
	\begin{figure}[h]
		\centering
		\def\svgwidth{0.9\textwidth}
		{\footnotesize 
\begingroup%
  \makeatletter%
  \providecommand\color[2][]{%
    \errmessage{(Inkscape) Color is used for the text in Inkscape, but the package 'color.sty' is not loaded}%
    \renewcommand\color[2][]{}%
  }%
  \providecommand\transparent[1]{%
    \errmessage{(Inkscape) Transparency is used (non-zero) for the text in Inkscape, but the package 'transparent.sty' is not loaded}%
    \renewcommand\transparent[1]{}%
  }%
  \providecommand\rotatebox[2]{#2}%
  \newcommand*\fsize{\dimexpr\f@size pt\relax}%
  \newcommand*\lineheight[1]{\fontsize{\fsize}{#1\fsize}\selectfont}%
  \ifx\svgwidth\undefined%
    \setlength{\unitlength}{467.40342508bp}%
    \ifx\svgscale\undefined%
      \relax%
    \else%
      \setlength{\unitlength}{\unitlength * \real{\svgscale}}%
    \fi%
  \else%
    \setlength{\unitlength}{\svgwidth}%
  \fi%
  \global\let\svgwidth\undefined%
  \global\let\svgscale\undefined%
  \makeatother%
  \begin{picture}(1,0.25966942)%
    \lineheight{1}%
    \setlength\tabcolsep{0pt}%
    \put(0,0){\includegraphics[width=\unitlength,page=1]{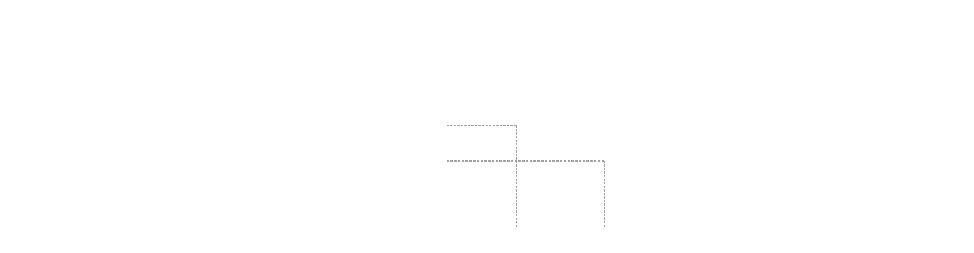}}%
    \put(0.38442233,0.21030866){\makebox(0,0)[lt]{\lineheight{1.25}\smash{\begin{tabular}[t]{l}$W_{pore}$\end{tabular}}}}%
    \put(0.98571899,0.00149763){\makebox(0,0)[lt]{\lineheight{1.25}\smash{\begin{tabular}[t]{l}$r$\\\end{tabular}}}}%
    \put(0.52123637,0.00149763){\makebox(0,0)[lt]{\lineheight{1.25}\smash{\begin{tabular}[t]{l}$r^*$\\\end{tabular}}}}%
    \put(0.60627731,0.00149763){\makebox(0,0)[lt]{\lineheight{1.25}\smash{\begin{tabular}[t]{l}$r_0$\\\end{tabular}}}}%
    \put(0.38442233,0.13237158){\makebox(0,0)[lt]{\lineheight{1.25}\smash{\begin{tabular}[t]{l}$E^*$\\\end{tabular}}}}%
    \put(0.38442233,0.08951362){\makebox(0,0)[lt]{\lineheight{1.25}\smash{\begin{tabular}[t]{l}$E_0$\\\end{tabular}}}}%
    \put(0,0){\includegraphics[width=\unitlength,page=2]{pores_and_energy.pdf}}%
  \end{picture}%
\endgroup%
 }
		\caption{Left: An intact membrane forms a hydrophobic defect, which transforms into a hydrophilic pore. Right: Schematic of $W_{pore}$ (dark blue, solid). The quadratic decay due to the surface tension (light blue, dashed) is disregarded.}
	\end{figure}
	\bigskip
	The second part of $W$ consists of the capacitive term $\tilde{C}_m  r^2 V_m^2$. It depends on the coefficient $\tilde{C}_m$, which represents the change in capacitance when water displaces lipids during pore formation, and the transmembrane potential $V_m$, which in turn depends on the applied external voltage $V_{ext}(t)$ and the pore distribution $n$. The membrane is modelled as a combination of a conductor and a capacitor.
	We consider the case of a planar bilayer lipid membrane placed between two equidistant electrodes, for which it holds
	\begin{equation}\label{eqn-electroporation-Vm}
		C_m \partial_t V_m + S_m V_m + 2\sigma_c V_m \int_{0}^{\infty} r n \d r = \frac{\sigma_c}{L} \big(- V_m + V_{ext}\big).
	\end{equation}
	Here $C_m$ is the membrane capacitance, $S_m$ the surface conductivity, $\sigma_c$ the conductivity of the liquid outside the membrane and $L$ the distance between the electrodes. The term $2\sigma_c V_m \int_{0}^{\infty} r n \d r$ is the current through conducting pores in the regime where $r$ is greater than the membrane thickness \cite{Poignard2017, chen_membrane_2006, krassowska_modeling_2007, kotnik_cell_2012}. We observe that there is a balancing effect between $V_m$ and the growth of the pores: when $V_m$ is large the pores grow more rapidly due to \eqref{rgrowth}, but an increase in the average pore size decreases $V_m$. 
	
	The study of (\ref{eqn-electroporation-n}) and (\ref{eqn-electroporation-Vm}) is  difficult due to  their dependence on constants which often cannot be measured experimentally, and are only known up to orders of magnitude even in the best case. Furthermore, the shape of the pore energy is based primarily on theoretical arguments, so that (\ref{eqn-electroporation-n}) offers in general only a qualitative description of the electroporation process \cite{Gianulis2017}. 
	
	We consider an approximate model describing the case when $\sigma_l$ is relatively small compared to the metastable pore radius $r_0$ and the timescale $T_{ext}$ of the external electrical field is much slower than the characteristic time $T^* = \frac{1}{D}e^{\frac{E^*}{k_B T}}$ required to pass the energy barrier in $W_{pore}.$ Without loss of generality we may take $k_B T = 1$. We split the problem into an interior region $r \in [0, r_0]$ and an outside region $r \geq r_0$ and first consider the case $V_m = 0$. For $r > r_0$ it holds $\partial_r W_{pore}(r) \sim 2 \pi \sigma_l$, so the characteristic pore size is $r_+ = \frac{1}{2\pi \sigma_l} > r_0$.
	Rescaling $s = \frac{D}{r_+^2}t$ and $x = \frac{r}{r_+}$ in the external region yields to leading order
	$$\partial_s n =  \partial_x \big( \partial_x n +  n \big)$$
	for $x \geq \frac{r_0}{r_+}$.
	In the interior we use the alternative scaling $y = \frac{r}{r_0}$ and obtain for $y \in [0,1]$
	\begin{equation*}
		\partial_s n_{in} = \frac{r_+^2}{r_0^2} \partial_y \left( \partial_y n_{in} +  \partial_y W_{pore}(y) n_{in} \right) + s(y)\Big(a_0e^{-\frac{W_{pore}(y)}{k_B T}} - n_{in}\Big)
		.
	\end{equation*}
	This converges to its equilibrium $n_{in}(y, t) \to a_0 e^{-W_{pore}(y)}$ much faster than the external problem since
	$\frac{r_+^2}{r_0^2} \gg 1$. We therefore assume $n_{in}(y,t) = a_0 e^{-W_{pore}(y)}$ and consider only the outside problem, for which we take at $x = \frac{r_0}{r_+}$ the Dirichlet boundary value $n_{in}(1) = a_0e^{-E_0}.$
	
	Now we consider the effect of $V_m \neq 0$. We assume that $V_m$ adjusts immediately to changes in $V_{ext},$ i.e. we obtain from \eqref{eqn-electroporation-Vm}
	that 
	$$V_m(s) = \frac{V_{ext}(s)}{\frac{L}{\sigma_c}S_m + 1 + 2L
		\Big(r_+^2 \int_{\frac{r_0}{r_+}}^{\infty} x n(x,s) \d x
		+ r_0^2 \int_{0}^{1} y a_0 e^{-W_{pore}(y)} \d y
		\Big)
		}.$$
	Thus, denoting $a_1 = 2Lr_0^2 \int_{0}^{1} y a_0 e^{-W_{pore}(y)} \d y$ and rescaling $f(x,s) = \frac{2L r_+^2}{\frac{L S_m}{\sigma_c} + 1 + a_1} n(x - \frac{r_0}{r_+}, s)$ it holds
	\begin{align*}
		V_m(s)^2 \approx \frac{V_{ext}(s)^2}{\big(\frac{L S_m}{\sigma_c} + 1 + a_1\big)^2 + \big(\frac{L S_m}{\sigma_c} + 1 + a_1 \big)^2
			\left(\int_{0}^{\infty} x f(x,s) \d x\right)^2
		}
	\end{align*}
	and thus
	\begin{align*}
		\partial_x W(x, V_m(s)) &\approx 1 -  2 \tilde{C}_m r_+^2 x V_m^2(s)
		\approx 1 - \frac{
			\frac{2r_+^2 \tilde{C}_m}{\left(\frac{L S_m}{\sigma_c} + 1 + a_1\right)^2} V_{ext}(s)^2 x
			}
			{1 + \left(\int_{0}^{\infty} x f(x,s) \d x\right)^2
			}.
	\end{align*}
	
	Taking $\beta = \frac{2r_+^2 \tilde{C}_m}{\left(\frac{L S_m}{\sigma_c} + 1 + a_1\right)^2}$ and $\mu = \frac{2L r_+^2}{\frac{L S_m}{\sigma_c} + 1 + a_1}a_0 e^{-E_0}$ and denoting the time variable again as $t$ yields
	\begin{equation}\label{eqn-main-general-vm}
	\left\{ 
	\begin{split}
		\partial_t f &= \partial_x \Big( \partial_x f + \Big(1 - \frac{\beta V_{ext}(t)^2}{1 + n_f^2 } x\Big)f   \Big), \quad x > 0 \\
		f(0, t) &= \mu.
	\end{split}
	\right.
\end{equation}
	We remark that in this model, the contribution of $V_m$ to pore growth lies in the reduction of the energy for large pores, not in the reduction of the energy barrier $E^*$, which we assume is small enough to have metastable hydrophilic pores even when $V_m = 0$. This can be compared to \cite{neu_asymptotic_1999}, who consider the situation where the energy barrier is very high and hydrophilic pores can only be formed when it is reduced by means of $V_m$. Furthermore, in our model $V_m$ depends explicitly on the pore distribution, in contrast to \cite{kavian_classical_2014}, who consider a variant of (\ref{eqn-electroporation-Vm}) with locally varying surface conductivity instead of pores. 
	
	\subsection{Main results}\label{Ss.results}
	
	As in \cite{krassowska_modeling_2007}, we  consider here the case of  constant potential. We assume from now on that  $V_{ext} = 1$, such that \eqref{eqn-main-general-vm} is equivalent to (\ref{eqn-main-parabolic}). Due to the effect of this constant potential we expect that the long-time behaviour of solutions to \eqref{eqn-main-parabolic} is determined by the corresponding transport problem 
	\begin{equation}\label{eqn-main-hyperbolic}
		\left\{ 
		\begin{split}
			\partial_t f &= \partial_x \Big(  \Big(1 - \frac{\beta}{1 +   n_f^2 } x\Big)f   \Big), \quad x > 0,\\
			f(x, 0) &= f_0.
		\end{split}
		\right. .
	\end{equation}
	
	We are now going to compute possible steady states in self-similar variables. Rescaling $f(x,t) = \frac{1}{t^{\frac{3}{2}}} F(y, \tau), y = \frac{x}{t}, \tau = \ln(t)$ yields for (\ref{eqn-main-parabolic})
	\begin{equation}\label{eqn-parabolic-selfsim-version}
		\partial_\tau F =  e^{-\tau}\partial_y^2 F +  \Big(1 + \Big(1 - \frac{\beta e^{\tau} }{1 +  e^\tau N_F^2}\Big) y \Big) \partial_y F + \Big(\frac{3}{2}  - \frac{\beta e^{\tau} }{1 +  e^\tau N_F^2} \Big) F, \quad \quad
		F(y, 0) = \mu e^{\frac{3}{2}\tau},
	\end{equation}
	and for the hyperbolic variant (\ref{eqn-main-hyperbolic})
	\begin{equation}\label{eqn-hyperbolic-selfsim-version}
		 \partial_\tau F =  \Big(1 + \Big(1 - \frac{\beta e^{\tau} }{1 +  e^\tau N_F^2}\Big) y \Big) \partial_y F + \Big(\frac{3}{2}  - \frac{\beta e^{\tau} }{1 +  e^\tau N_F^2} \Big) F,
	\end{equation}
	where in both cases
	$$N_F = \int_{0}^{\infty} yF(y, \tau) \d y = t^{-\frac{1}{2}} \int_{0}^{\infty} x f \d x =  n_f t^{-\frac{1}{2}}.$$
	For large times, assuming that $ N_F$ does not vanish it holds $\frac{\beta e^\tau y}{1 + e^\tau N_F^2} \approx \frac{\beta y}{ N_F^2},$
	so that (\ref{eqn-hyperbolic-selfsim-version}) admits asymptotic steady states. We denote these as $F_s(y)$ and they solve
	\begin{equation}\label{eqn-Fs}
		0 = \left(1 +  \gamma y\right) \partial_y F_s + \Big(\gamma + \tfrac{1}{2}\Big)F_s,
	\end{equation}
	where
	\begin{equation}\label{eqn-def-gamma}
		\gamma \coloneqq 1 - \frac{\beta}{ N_s^2}, \hspace*{1.2cm} N_s \coloneqq \int_{0}^{\infty} y F_s(y) \d y .
	\end{equation}
	
	Solutions to \eqref{eqn-Fs}-\eqref{eqn-def-gamma} can be easily computed explicitly. We summarize the results in the following Proposition \ref{P.fs}
	
	\bigskip
	\begin{proposition}\label{P.fs}
		The solutions of (\ref{eqn-Fs}) with finite mass $N_s$ are a one-parameter family of functions parametrized by $\gamma \in \left( -\infty, \frac{1}{2} \right)$ and are given by
		\begin{enumerate}[(i)]
			\item $\gamma \in \left(0, \frac{1}{2}\right):$ 
			$$F_s(y) = c_s(1 + \gamma y)^{-(1 + \frac{1}{2\gamma})}, \qquad c_s = \frac{1 - 2\gamma}{4}N_s, \qquad N_s = \Big(\frac{\beta}{ (1 {-}\gamma)} \Big)^{\frac{1}{2}}$$
			\item $\gamma = 0:$
			$$F_s(y) = c_s e^{-\frac{y}{2}}, \qquad c_s = \frac{\beta^{\frac{1}{2}}}{4}, \qquad N_s = 4c_s$$
			\item $\gamma < 0, \gamma \neq -\frac{1}{2}:$ 
			$$F_s(y) = c_s(1 + \gamma y)^{-(1 + \frac{1}{2\gamma})}\chi_{\left\{y < -\frac{1}{\gamma}\right\}}, \qquad c_s = \frac{1 - 2\gamma}{4}N_s, \qquad N_s = \Big(\frac{\beta}{ (1 {-} \gamma)} \Big)^{\frac{1}{2}}$$
			where $\chi$ is the indicator function, i.e. $F_s$ has bounded support. 
			\item $\gamma = -\frac{1}{2}:$
			$$F_s = c_1 \chi_{\{y < 2\} }+ c_2 \delta_{ \{ y = 2\} }$$
		\end{enumerate}
	\end{proposition}
	
	\bigskip
	Our goal in this paper  is to study whether the solutions to  (\ref{eqn-main-parabolic}), resp. (\ref{eqn-parabolic-selfsim-version}) converge to any of these asymptotic self-similar solutions and which one is selected. The difficulty in the analysis  is that, unless $V_{ext} =0$, there is an external supply of energy, so that the system does not have a decreasing free energy that could be used as a Lyapunov functional. Furthermore, there are no conserved moments and due to the nonlocal term, comparison principles are not directly available. Thus 
	 a global stability results seem difficult to obatin. We will show here a local stability result, that is if the initial data are in a certain sense close to the self-similar power-law solution $F_s$, then $F \to F_s$ for $y > 0$ as $\tau \to \infty$. Our proof is based on comparison to solutions of  the transport equation (\ref{eqn-hyperbolic-selfsim-version}).
	 
	\bigskip
	We remark here that \eqref{eqn-main-hyperbolic} is somewhat reminiscent of the classical LSW model for Ostwald Ripening which reads $\partial_t f + \partial_x \big( (\theta(t) x^{1/3} -1 ) f\big) =0$ where the mean-field $\theta$ is such that the first moment of $f$ is conserved. This model also has a one-parameter family of self-similar solutions all with compact support and a certain power law  or exponential behaviour respectively there. The long-time behaviour of solutions is then determined by the behaviour of the initial data at the end of their support \cite{NP99,NV06a,NV06b}. In contrast to the results in this paper for \eqref{eqn-main-parabolic}, if one adds an additional second order term in the LSW equation, then the long-time behaviour changes completely. Indeed, one expects, even though nothing is rigorously proved, that the only possible limit in self-similar variables is the solution with exponential decay. 
	
	In a related LSW-type model, analyzed in \cite{Carr_asymptotic_1998}, the term $x^{1/3}$ in the LSW model is replaced by $x$ and is thus closer to \eqref{eqn-main-hyperbolic}. In this case the analysis of the equation is somewhat simpler since it can  basically be  solved explicitly. Indeed, even though \eqref{eqn-main-hyperbolic} does not have a conserved first moment, the methods of \cite{Carr_asymptotic_1998}  can be adapted and the long-time behaviour of \eqref{eqn-main-hyperbolic} can be analyzed using these methods, which are based on characteristics. Here, we choose however a different approach based on a linearization of the equation. The reason for this is, that our main goal is to analyze the parabolic equation \eqref{eqn-main-parabolic}, for which the method of characteristics does not apply. Instead, we show a comparision principle for the linearized equation, for which we need the analysis of the corresponding hyperbolic equation \eqref{eqn-main-hyperbolic} as a stepping stone. 
	
	\bigskip
		In the present paper we restrict to the case that $F_s$ decays as a power law, i.e. the case $\gamma \in (0,\frac 1 2)$ and denote
	\begin{equation}\label{def-theta}
		\theta \coloneqq  1 + \frac{1}{2\gamma} > 2.
	\end{equation} 
For our local stability result we consider an initial time $\tau_0$ which needs to be sufficiently large and assume that $F(y,\tau_0)=F_s(y)+G(y,\tau_0)$ where $G(\cdot, \tau_0)$ decays faster than
$F_s$ and has certain regularity properties (see \eqref{thm-main-hyperbolic-ass-G0} and 
\eqref{thm-main-parabolic-ass-G0} below). The assumption that $\tau_0$ is large enough means that we assume that the solution is initially spread across a sufficiently large domain. 

	\bigskip
	Our first main result is now the following local stability result for 
	 the transport equation \eqref{eqn-hyperbolic-selfsim-version}.
	
	\bigskip
	\begin{theorem}\label{thm-main-hyperbolic}
		Let $F$ be the solution of (\ref{eqn-hyperbolic-selfsim-version}) and $F_s$ the solution of (\ref{eqn-Fs}) for a given $\gamma \in \left(0, \frac{1}{2}\right)$. Let $G = F - F_s$ have mass $N_G = \int_{0}^{\infty} yG(y, \tau) \d y$ and let
		$\varepsilon \in \left(0, \frac{1-\gamma}{\gamma}\right)$. 
		Then there exist $\bar{c}_0, \bar{\tau}_0 > 0$ and $c(\gamma, \varepsilon) > 0$ such that if $G(y, \tau_0) = G_0(y) + \hat{G}_0(y)$,
		\begin{equation}\label{thm-main-hyperbolic-ass-G0}
		|G_0(y)| \leq c_0\left(1 + \gamma y\right)^{-\theta - \varepsilon},
		\quad
		|\hat{G}_0(y)| \leq e^{(2 - \varepsilon \gamma)\tau_0 - \frac{ye^{\tau_{0}}}{2}}
		\end{equation}
		for $\tau_0 \geq \bar{\tau}_0$ and $c_0 \leq \bar{c}_0$, then
		$$|N_G(\tau)| \leq c(\gamma, \varepsilon)\left(e^{-\varepsilon \gamma \tau_0 } + c_0 \right) e^{- \varepsilon \gamma(\tau - \tau_0)}$$
		for all $\tau \geq \tau_0$, i.e. $N_F \to N_s$ exponentially as $\tau \to \infty$. In addition, it holds
		$F(y, \tau) \to F_s(y)$ as $\tau \to \infty$ uniformly on compact sets in $[0, \infty)$. 
	\end{theorem}
	
	\bigskip
	Here, $G_0$ controls the tail while $\hat{G}_0$ controls the region $x \approx 1$, where the initial data can be much larger. Theorem \ref{thm-main-hyperbolic} serves as preparation for the parabolic case, for which we obtain a similar result with the exception that $F(y,\tau) - F_s(y) \to 0$ as $\tau \to \infty$ holds only for $y > 0$. This is because $f$ has a boundary layer at $x \approx 1$ due to the Dirichlet boundary condition, so that $F$ has a boundary layer for $y \approx \frac{1}{t}$. This is reflected in the theorem by the presence of the term $\mu e^{\frac{3}{2}\tau - ye^{\tau}}$ in the limit of $F$, which corresponds to $\mu e^{-x}$ in the limit of $f$, and is the reason we allow inital data that are potentially very large in the region $x \approx 1$. 
	
	\bigskip
	
	\begin{theorem}\label{thm-main-parabolic}
		Let $F$ be the solution of (\ref{eqn-parabolic-selfsim-version}) and $F_s$ the solution of (\ref{eqn-Fs}) for a given $\gamma \in \left(0, \frac{1}{2}\right)$.
		Let $G = F - F_s$ have mass $N_G = \int_{0}^{\infty} yG(y, \tau) \d y$.
		Then given
		 $ \varepsilon \in \left(0, \frac{1}{2\gamma} \right]$,
		there exist $\bar{c}_0, \bar{\tau}_0 > 0$ and $c(\gamma, \varepsilon) > 0$ such that if $G(y, \tau_0) = G_0 + \hat{G}_0$, 
		\begin{equation}\label{thm-main-parabolic-ass-G0}
			|G_0(y)| \leq c_0\left(1 + \gamma y\right)^{-\theta - \varepsilon}, \quad
			|\partial_y^2 G_0(y)| \leq c_0\left(1 + \gamma y\right)^{-\theta - 2}, \quad
			|\hat{G}_0(y)| \leq e^{(2 - \varepsilon \gamma)\tau_0 - \frac{ye^{\tau_{0}}}{2}}
		\end{equation}
		for $\tau_0 \geq \bar{\tau}_0$ and $c_0 \leq \bar{c}_0$, then
		$$|N_G(\tau)| \leq c(\gamma, \varepsilon)\left(e^{-\varepsilon \gamma \tau_0 } + c_0 \right) e^{- \varepsilon \gamma(\tau - \tau_0)}$$
		for all $\tau \geq \tau_0$, i.e. $N_F \to N_s$ exponentially as $\tau \to \infty$, and
		\begin{equation*}
			\lim\limits_{\tau \to \infty} 
			\left|
			\frac{F - F_s - \mu e^{\frac{3}{2}\tau - ye^{\tau}}}{ e^{\frac{3}{2}\tau - ye^{\tau}} + (1 + \gamma y)^{-\theta}}
			\right| \to 0
		\end{equation*}
		uniformly on compact sets in $[0, \infty)$. 
	\end{theorem}
	
	\bigskip
	Our strategy for the proof of Theorem \ref{thm-main-hyperbolic} and Theorem \ref{thm-main-parabolic} is to consider the term $\frac{1}{1 + N_F^2}$ in (\ref{eqn-main-parabolic}) as a perturbation of $\frac{1}{N_s^2}$ when $N_F \approx N_s$, which is true on a small time interval when $F$ is sufficiently close to $F_s$. We derive an integral equation for $N_F(\tau)-N_s(\tau)$ (Prop. \ref{prop-hyperbolic_NG}, Prop. \ref{prop-ng-parabolic}). We first consider the hyperbolic case, where studying the solutions of the perturbation formulation of (\ref{eqn-main-hyperbolic}) by characteristics allows us to identify the key terms and bound errors in the integral equation (Corollary \ref{lemma-hyperbolic-M}), showing by a continuation argument that $N_F(\tau)-N_s(\tau)$ remains small for all times (Lemma \ref{lemma-ODE-in-selfsim}). Inserting the estimates for $N_F(\tau)-N_s(\tau)$ into the solution by characteristics then yields the limit of $F$ (Lemma \ref{lemma-F-to-Fs}).
	
	With these results we treat the parabolic case, which carries the additional difficulty of understanding the solutions of the perturbative formulation of (\ref{eqn-main-parabolic}). We achieve this by an approximation lemma relating the solutions of the parabolic problem to the hyperbolic, for which the needed estimates are already proven (Lemma \ref{approx-lemma}). Additionally, we give an approximation of the boundary layer ensuring that $F(0,\tau) = e^{\frac{3\tau}{2}}\mu$.

	\bigskip
	In the following we will often use the key relations
	\begin{align}
		\frac{\beta}{N_s^2} = 1 {-} \gamma, \qquad
		\theta {-}1 = \frac{1}{2\gamma}, \qquad
		c_s = N_s \gamma^2 (\theta {-} 1)(\theta {-} 2),
	\end{align}
	which are easily checked using (\ref{eqn-def-gamma}) and (\ref{def-theta}).
	However, throughout the paper we do not explicitly track all positive constants depending only on $\beta, \mu, \gamma, \theta$ or $N_s$. Instead we use $c(\gamma)$ to denote such constants, which may change from line to line. 
	If a constant depends also on parameters introduced later, we write e.g. $c(\gamma, \varepsilon).$ We will work both in the original variables $(x,t)$ and the self-similar variables $(y, \tau)$, since the latter are more convenient in the hyperbolic case but the former are more natural for approximating solutions of (\ref{eqn-main-parabolic}) by solutions of (\ref{eqn-main-hyperbolic}). Functions in $x$ and $t$ are denoted by lowercase letters, while functions in the self-similar scaling are uppercase, i.e. $\varphi(x,t) = \frac{1}{t^{\frac{3}{2}}}\Phi(y,\tau)$.

	\section{The hyperbolic problem}
	The goal of this section is to prove \Cref{thm-main-hyperbolic} as preparation for \Cref{thm-main-parabolic}, including some key estimates on solutions of the transport equation when $N_G(\tau)$ is sufficiently small. To this end, let the assumptions of \Cref{thm-main-hyperbolic} hold and let
	$$f(x,t) = \frac{1}{t^{\frac{3}{2}}} F(y, \tau), \qquad 
	f_s(x, t) = t^{-\frac{3}{2}} F_s\big(\frac{x}{t}\big), \qquad 
	g(x,t) = \frac{1}{t^{\frac{3}{2}}} G(y, \tau)
	$$ 
	be the corresponding functions in original variables with mass $n_f = N_F t^{\frac{1}{2}}, n_s = N_s t^{\frac{1}{2}}$ and \linebreak $n_g = N_G t^{\frac{1}{2}}$, respectively.
	Our strategy for proving smallness of $n_g$ (resp. $N_G$) is understanding the ODEs they solve, which we derive in the next section.
	
	\subsection{Derivation of the integral equations for $n_g$ and $N_G$}\label{subsection-hyperbolic-derivation-ng}
	Our goal in this section is to derive an integral equation for $n_g$ (resp. $N_G$). To do so, we treat $n_g$ and therefore $n_f$ as  given functions and examine the resulting solutions of (\ref{eqn-main-hyperbolic}).
	
	\begin{proposition}\label{prop-ng}
		Let $t_0 > 0$ and denote $g(x,t_0) = g_0(x) + \hat{g}_0(x),$ $g_0(x) = \frac{1}{t_0^{\frac{3}{2}}} G_0(\frac{x}{t_0})$, \linebreak $\hat{g}_0(x) = \frac{1}{t_0^{\frac{3}{2}}} \hat{G}_0(\frac{x}{t_0}).$
		Define 
		\begin{align}
			\label{eqn-def-h}
			h(n_g, t) &\coloneqq \frac{1  + n_g^2 - 2n_g n_s^{-1}\left(1 + 2n_g n_s + n_g^2 \right)}{n_s^4\left(1 + \frac{2n_g n_s + n_g^2 + 1}{n_s^2} \right)} \\
			\label{eqn-def-j}
			j(t) &\coloneqq \frac{2\beta n_g }{n_s^3}
			+ \beta h(n_g, t) \\
			\label{eqn-def-mj}
			m(t, t_0) &\coloneqq \int_{t_0}^{t} \exp \Big(
			\int_{t_0}^{r} 
			-\frac{\beta}{n_s^2(\omega)} + j(\omega) 
			\d \omega \Big) \d r 
			= \int_{t_0}^{t} \Big( \frac{r}{t_0}\Big)^{-(1-\gamma)} \exp \Big(\int_{t_0}^{r} j(\omega) \Big) \d r
		\end{align}
		and define $T(t,t_0)\varphi_0$ as the solution mapping of 
			\begin{equation}\label{eqn-pde-for-T}
			\left\{ 
			\begin{split}
				\partial_t \varphi &= \partial_x \Big( \Big(1 + \beta x \Big( -\frac{1}{n_s^2} + j(t) \Big)\Big)\varphi   \Big), \quad x > 0 \\
				\varphi(x, t_0) &= \varphi_0(x).
			\end{split}
			\right.
		\end{equation}
		Then
		\begin{align}
			\label{eqn-def-Tj} 
			T(t,t_0)\varphi(x) &= m'(t, t_0) \varphi(m'(t, t_0)x + m(t, t_0)) 
			\\
			\label{eqn-g-duhamel-hyperbolic}
			g(x, t) &= T(t, t_0)(g_0 + \hat{g}_0)
			+ \int_{t_0}^{t} 
			\beta \Big(\frac{2n_g(r) }{n_s^3(r)}
			+ h(n_g(r), r)\Big)  T(t, r)  \partial_x(xf_s(x,r)) 
			\d r
		\end{align}
		and
		\begin{equation}\label{eqn-ng-duhamel-hyperbolic}
			\begin{split}
				n_g(t) &= \int_{0}^{\infty} x T(t, t_0)(g_0 + \hat{g}_0) \d x \\ 
				&\quad+ \int_{t_0}^{t} 
				\beta \Big(\frac{2n_g(r) }{n_s^3(r)}
				+ h(n_g(r), r)\Big) \int_{0}^{\infty} x T(t, r)  \partial_x(xf_s(x,r)) \d x
				\d r.
			\end{split}
		\end{equation}
	\end{proposition}
	We will occasionally denote $m(t) = m(t, t_0), m'(t) = m'(t, t_0)$ when the choice of $t_0$ is clear. 
	\begin{proof}
		It holds
		\begin{align*}
			\frac{1}{1 +  n_f^2 }
			&=  \frac{1}{n_s^2 }
			- \frac{1 +  n_f^2  - n_s^2}{n_s^2 (1 +  n_f^2)}
			= \frac{1}{n_s^2 } 
			- \frac{1 + 2n_g n_s  + n_g^2}{n_s^4 \Big(1 + \frac{2n_g n_s  + n_g^2 + 1}{n_s^2 } \Big)} \\
			&= \frac{1}{n_s^2 } - \frac{2n_g }{n_s^3}
			- h(n_g, t),
		\end{align*}
		so (\ref{eqn-main-hyperbolic}) can be rewritten to
		\begin{equation*}
			\partial_t f = \partial_x \Big(  \Big(1 + \beta x \Big( -\frac{1}{n_s^2} + \frac{2n_g }{n_s^3}
			+ h(n_g, t) \Big)\Big)f   \Big).
		\end{equation*}
		
		Since $f_s$ solves
		\begin{equation*}
			\partial_t f_s = \partial_x \Big( \Big(1 - \frac{\beta x}{n_s^2}\Big)f_s\Big)
		\end{equation*}
		it follows that
		\begin{equation}\label{eqn-diffeq-g-hyperbolic}
			\begin{split}
				\partial_t g 
				&= \Big(1 + \beta x \Big(-\frac{1}{n_s^2} + \frac{2n_g }{n_s^3}
				+ h(n_g, t) \Big)  \Big)\partial_x g
				+ \beta \Big(-\frac{1}{n_s^2} + \frac{2n_g }{n_s^3}
				+ h(n_g, t) \Big) g \\
				&\quad + \partial_x \Big( \Big(\frac{2n_g }{n_s^3}
				+ h(n_g, t)\Big) \beta x f_s  \Big).
			\end{split}
		\end{equation}
		\Cref{eqn-g-duhamel-hyperbolic} follows by Duhamel's principle, and (\ref{eqn-ng-duhamel-hyperbolic}) by integration. Finally, it holds
		\begin{align*}
			m(t, t_0) &= \int_{t_0}^{t} \exp \Big(
			\int_{t_0}^{r} 
			-\frac{\beta}{N_s^2 \omega } + j(\omega) 
			\d \omega \Big) \d r 
			= \int_{t_0}^{t} \Big( \frac{r}{t_0}\Big)^{-(1-\gamma)} \exp \Big(\int_{t_0}^{r} j(\omega) \Big) \d r 
		\end{align*}
		since $N_s$ is a constant. 
	\end{proof}
It is occasionally more convenient to work in self-similar variables, where we obtain a similar result:
\begin{corollary}\label{prop-hyperbolic_NG}
	Let $\tau_0 > 0$ and $G$ be as in Theorem \ref{thm-main-hyperbolic}.
	Define 
	\begin{align}
		\label{eqn-def-H}
		H(N_G, \tau) &\coloneqq \frac{ e^{-\tau} + N_G^2 - 2N_G N_s^{-1} (e^{-\tau} + 2N_sN_G + N_G^2) }{N_s^4 \Big(1 + \frac{e^{-\tau} + 2N_sN_G + N_G^2}{N_s^2} \Big)}
		= t h(n_g(t), t) \\
		\label{eqn-def-J}
		J(\tau) &\coloneqq \frac{2\beta N_G}{N_s^3} + \beta H(N_G, \tau)
		= tj(t) \\
		\label{eqn-def-MJ}
		M(\tau, \tau_0) &\coloneqq \int_{\tau_0}^{\tau} \exp \Big( \int_{\tau_0}^{r} \gamma + J(\omega) \d \omega \Big) \d r
		= t_0^{-1}m(t, t_0)
	\end{align}
	and define $\tilde{T}(\tau, \tau_0)$ as the solution mapping of
		\begin{equation*}
		\left\{ 
		\begin{split}
			\partial_\tau \Phi &= \Big(1 + \big(\gamma + \beta J(\tau)\big)\Big)\partial_y\Phi + \Big(\frac{3}{2} + \gamma + J(\tau) \Big)\Phi,
			 \quad y > 0 \\
			\Phi(y, \tau_0) &= \Phi_0(y).
		\end{split}
		\right.
	\end{equation*}
	Then
	\begin{align}
		\label{eqn-def-tildeTJ} 
		\tilde{T}(\tau,\tau_0)\Phi(y) &= e^{\frac{1}{2}(\tau - \tau_0)}M'(\tau, \tau_0) \Phi(M'(\tau, \tau_0)y + M(\tau, \tau_0)) \\
		&= t^{\frac{3}{2}} T(t, t_0)\varphi(x) 
		\hspace*{0,5cm} \text{for } \varphi(x) = t_0^{-\frac{3}{2}} \Phi\Big(\frac{x}{t_0}\Big) \nonumber \\
		\label{eqn-G-duhamel-hyperbolic-selfsim}
		G(y, \tau) &= \tilde{T}(\tau, \tau_0)(G_0(y) + \hat{G}_0(y)) \\
			&\quad+ \int_{\tau_0}^{\tau} 
			\beta \Big( \frac{2N_G(r)}{N_s^3} + H(N_G, r)\Big)  \tilde{T}(\tau, r)  \partial_y(yF_s(y,r)) 
			\d r
	\end{align}
	and
	\begin{equation}\label{eqn-NG-duhamel-hyperbolic-selfsim}
		\begin{split}
			N_G(\tau) &= \int_{0}^{\infty} y \tilde{T}(\tau, \tau_0)(G_0(y) + \hat{G}_0(y)) \d y \\
			&\quad + \int_{\tau_0}^{\tau} 
			\beta \Big( \frac{2N_G(r)}{N_s^3} + H(N_G, r)\Big) \int_{0}^{\infty} y \tilde{T}(\tau, r)  \partial_y(yF_s(y,r))  \d y
			\d r.
		\end{split}
	\end{equation}
\end{corollary}
\begin{proof}
	Analogously to \Cref{prop-ng} we compute
		\begin{align*}
				  \frac{e^\tau}{1 + e^\tau N_F^2}
				&= \frac{1}{N_s^2} 
				- \frac{e^{-\tau} + N_F^2 - N_s^2}{N_s^4 \Big(1 + \frac{e^{-\tau} + N_F^2 - N_s^2}{N_s^2} \Big)}
				= \frac{1}{N_s^2}
				- \frac{e^{-\tau} + 2N_sN_G + N_G^2}{N_s^4 \Big(1 + \frac{e^{-\tau} + 2N_sN_G + N_G^2}{N_s^2} \Big)} \\
				&= \frac{1}{N_s^2}
				-\frac{2N_G}{N_s^3} - H(N_G, \tau)
			\end{align*}
	and
		\begin{align*}
				\partial_\tau G &= 
				\Big( 1 + \Big( 1 - \frac{\beta}{N_s^2} + \frac{2\beta N_G}{N_s^3} + \beta H(N_G, \tau)\Big) y\Big) \partial_y G
				+
				\Big(\frac{3}{2} - \frac{\beta}{N_s^2} + \frac{2\beta N_G}{N_s^3} + \beta H(N_G, \tau)  \Big) G 
				\\
				&\quad +  \Big( \frac{2\beta N_G}{N_s^3} + \beta H(N_G, \tau) \Big) \partial_y \big( yF_s \big)
				\\
				&= \Big( 1 + \Big( \gamma + \frac{2\beta N_G}{N_s^3} + \beta H(N_G, \tau)\Big) y\Big) \partial_y G
				+
				\Big(\frac{1}{2} + \gamma + \frac{2\beta N_G}{N_s^3} + \beta H(N_G, \tau)  \Big) G 
				\\
				&\quad +  \Big( \frac{2\beta N_G}{N_s^3} + \beta H(N_G, \tau) \Big) \partial_y \big( yF_s \big)
			\end{align*}
	Again (\ref{eqn-G-duhamel-hyperbolic-selfsim}) follows by Duhamel's principle and (\ref{eqn-NG-duhamel-hyperbolic-selfsim}) by integration. The relations between $h, j, T$ and $m$ and their self-similar scaling counterparts $H, J, \tilde{T}$ and $M$ are easily checked.
\end{proof}

\subsection{Estimates for the solution of the hyperbolic equation}
To understand the integral equations (\ref{eqn-ng-duhamel-hyperbolic}) and (\ref{eqn-NG-duhamel-hyperbolic-selfsim}), we study the operators $T$ and $\tilde{T}$ under suitable smallness assumptions on $j$ and $J$, respectively. 

	\begin{lemma}\label{lemma-hyperbolic-m}
		Let $t_0 \geq 1$ and $j$ as in (\ref{eqn-def-j}) satisfy $|j(t)| \leq c_j t_0^\eta t^{-1 - \eta}$ for $c_j, \eta > 0$. Define $T(t, t_0)$ and $m(t, t_0)$ as in $(\ref{eqn-def-Tj})$ and $(\ref{eqn-def-mj})$. Then for all $t \geq t_0$ it holds
		\begin{enumerate}[(i)]
			\item $c_m^{-1}\Big(\frac{t}{t_0} \Big)^{-(1-\gamma)} \leq m'(t, t_0) \leq c_m \Big(\frac{t}{t_0} \Big)^{-(1-\gamma)}$
			 for $c_m = e^{\frac{c_j}{\eta}} \geq 1,$
			\item $c_m^{-1}\gamma^{-1}t_0\Big( \Big(\frac{t}{t_0} \Big)^{\gamma } - 1 \Big) \leq m(t, t_0) \leq c_m\gamma^{-1}t_0\Big( \Big(\frac{t}{t_0} \Big)^{\gamma } - 1 \Big),$
			\item and
			\begin{equation*}
				\frac{2\beta}{N_s^3}\int_{0}^{\infty} xT(t,r)\partial_x(xf_s(x,r)) \d x 
				= -\frac{1 - \gamma}{\gamma} t^{\frac{1}{2}}
				\Big(
				1 
				- (1 - 2 \gamma) \Big(\frac{t}{r} \Big)^{-\gamma}
				+ \delta_1(r)
				\Big)
			\end{equation*}
			for $|\delta_1(r)| \leq c(\gamma, \eta)c_j.$ 
		\end{enumerate}
	\end{lemma}

	\begin{proof}
		\begin{enumerate}[(i)]
		\item 
		It holds 
		\begin{align*}
			\Big|\int_{t_0}^{t}  j(r) \d r \Big|
			&\leq c_j t_0^\eta \int_{t_0}^{t} r^{-1 - \eta} \d r
			\leq \frac{c_j}{\eta},
		\end{align*}
		which combined with (\ref{eqn-def-mj}) and $c_m = e^{\frac{c_j}{\eta}}$ yields the desired bound.
		
		\item Follows from $(i)$.
		
		\item For $\varphi\in L^1((1 + x) \d x)$ it holds
		$$\int_{0}^{\infty} xT(t,r)\varphi(x) \d x = \frac{1}{m'(t,r)}\int_{m(t,r)}^{\infty} x\varphi(x) \d x - \frac{m(t,r)}{m'(t,r)}\int_{m(t,r)}^{\infty} \varphi(x) \d x,$$
		so we compute
		$$\int_{m(t,r)}^{\infty} \partial_x(xf_s(x,r)) \d x = -m(t,r)c_s r^{-\frac{3}{2}} 
		\Big(1 + \frac{\gamma m(t,r)}{r}\Big)^{-\theta}.$$
		Using
		$$\int_{m(t,r)}^{\infty} x\Big(1 + \frac{\gamma x}{r}\Big)^{-\theta} \d x
		= r^2 \gamma^{-2} \Big(-\frac{1}{\theta - 1}\Big(1 + \frac{\gamma m(t,r)}{r}\Big)^{-(\theta - 1)}  + \frac{1}{ \theta - 2}\Big(1 + \frac{\gamma m(t,r)}{r}\Big)^{-(\theta - 2)} \Big)$$
		we estimate
		\begin{align*}
			\int_{m(t,r)}^{\infty} x \partial_x (xf_s) \d x
			&= c_s r^{\frac{1}{2}} \gamma^{-2} 
			\Big(\frac{1}{\theta - 1}\Big(1 + \frac{\gamma m(t,r)}{r}\Big)^{-(\theta - 1)} 
			- \frac{1}{\theta - 2}\Big(1 + \frac{\gamma m(t,r)}{r}\Big)^{-(\theta - 2)} \Big) \\
			&\quad - m(t,r)^2c_s r^{-\frac{3}{2}}\Big(1 + \frac{\gamma m(t,r)}{r}\Big)^{-\theta}.
		\end{align*}
		All together this yields
		\begin{align*}
			&\quad \; \int_{0}^{\infty} xT(t,r)\partial_x(xf_s(x,r)) \d r \\
			&= \frac{c_s}{m'(t,r)} r^{\frac{1}{2}} \gamma^{-2}
			\Big(  
			-\frac{1}{\theta - 2} \Big(1 + \frac{\gamma m(t,r)}{r}\Big)^{- (\theta - 2)}
			+ \frac{1}{\theta - 1}  
			\Big(1 + \frac{\gamma m(t,r)}{r}\Big)^{ - (\theta - 1)}
			\Big).
		\end{align*}
		Using $(ii)$ and $c_m^{-1} \leq 1$ we have for any exponent $k > 0$ that
		$$\Big(1 + \frac{\gamma m(t,r)}{r}\Big)^{-k}
		\leq \Big(1 - c_m^{-1} + c_m^{-1}\Big(\frac{t}{r} \Big)^{\gamma} \Big)^{-k} \leq c_m^{k}\Big(\frac{t}{r} \Big)^{-\gamma k}
		$$
		and
		$$\Big(1 + \frac{\gamma m(t,r)}{r}\Big)^{-k}
		\geq \Big(1 - c_m + c_m\Big(\frac{t}{r}\Big)^{\gamma} \Big)^{-k} \geq c_m^{-k}\Big(\frac{t}{r}\Big)^{-\gamma k}.
		$$
		So keeping in mind $\theta > 2$, $\frac{\beta}{N_s^2} = 1 - \gamma$ and $c_s = N_s \gamma^2 (\theta - 1)(\theta - 2)$ we calculate
		\begin{align*}
			&\quad \frac{2\beta}{N_s^3} \int_{0}^{\infty} xT(t,r)\partial_x(xf_s(x,r)) \d r \\
			&= 2(1 - \gamma)(\theta - 1)(\theta - 2) \frac{r^{\frac{1}{2}}}{m'(t,r)}
			\Big(  
			-\frac{1}{\theta - 2} \Big(1 + \frac{\gamma m(t,r)}{r}\Big)^{- (\theta - 2)}
			+ \frac{1}{\theta - 1}  
			\Big(1 + \frac{\gamma m(t,r)}{r}\Big)^{ - (\theta - 1)}
			\Big)
			\\
			&\leq 2(1 - \gamma)(\theta - 1)(\theta - 2) c_m r^{\frac{1}{2}} \Big(\frac{t}{r} \Big)^{\gamma - 1}
			 \Big(
			-\frac{c_m^{- \theta + 2}}{\theta - 2} \Big(\frac{t}{r} \Big)^{\gamma (- \theta + 2)}
			+ \frac{c_m^{\theta - 1}}{\theta - 1}\Big(\frac{t}{r} \Big)^{\gamma (- \theta + 1)}
			\Big) \\
			&= -\frac{1 - \gamma}{\gamma} t^{\frac{1}{2}} 
			\Big(  
			c_m^{3-\theta}
			- c_m^{\theta} \frac{\theta - 2}{\theta - 1}\Big(\frac{t}{r} \Big)^{-\gamma}
			\Big) 
			= -\frac{1 - \gamma}{\gamma} t^{\frac{1}{2}} 
			\Big(  
			c_m^{3-\theta}
			- c_m^{\theta} (1 - 2\gamma)\Big(\frac{t}{r} \Big)^{-\gamma}
			\Big)
		\end{align*}
		and analogously
		\begin{align*}
			\frac{2\beta}{N_s^3} \int_{0}^{\infty} xT(t,r)\partial_x(xf_s(x,r)) \d r 
			&\geq
			-\frac{1 - \gamma}{\gamma} t^{\frac{1}{2}} 
			\Big(  
			c_m^{-3+\theta}
			- c_m^{-\theta} (1 - 2\gamma)\Big(\frac{t}{r} \Big)^{-\gamma}
			\Big)
		\end{align*}
		Thus we have
		\begin{align*}
			\frac{2\beta}{N_s^3}\int_{0}^{\infty} xT(t,r)\partial_x(xf_s(x,r)) \d r 
			&=  -\frac{1 - \gamma}{\gamma} t^{\frac{1}{2}}
			\Big(
			1 
			- (1 - 2\gamma) \Big(\frac{t}{r} \Big)^{-\gamma}
			+ \delta_1(r)
			\Big),
		\end{align*}
		where using Taylor's formula we obtain
		\begin{align*}
			|\delta_1(r)| 
			&\leq \max \big\{ \big| c_m^{-3+\theta} - 1 \big|, \big| c_m^{3-\theta} - 1 \big|\big\}
			+ \max \big\{ \big| c_m^{-\theta} - 1 \big|, \big| c_m^{\theta} - 1 \big|\big\} 
			\leq \frac{2 \theta }{\eta} c_j c_m^\theta \leq c(\gamma, \eta) c_j.
		\end{align*}
		\end{enumerate}
	\end{proof}
	
The corresponding result in the self-similar scaling reads as follows.
\begin{corollary}\label{lemma-hyperbolic-M}
	Let  $\tau_0 \geq 0$ and $J$ as in (\ref{eqn-def-J}) satisfy $|J(\tau)| \leq c_J e^{-\eta (\tau - \tau_0)}$ for $c_j, \eta > 0$. Define $\tilde{T}(\tau, \tau_0)$ and $M(\tau, \tau_0)$ as in $(\ref{eqn-def-tildeTJ})$ and $(\ref{eqn-def-MJ})$. Let $c_m$ be as in \Cref{lemma-hyperbolic-m} with $c_J$ replacing $c_j$. Then for all $\tau \geq \tau_0$ it holds
	\begin{enumerate}[(i)]
		\item $c_m^{-1}e^{\gamma (\tau - \tau_0)} \leq M'(t, t_0) \leq c_m e^{\gamma (\tau - \tau_0)},$
		\item $c_m^{-1}\gamma^{-1}\big( e^{\gamma (\tau - \tau_0)} - 1 \big) \leq M(t, t_0) \leq c_m\gamma^{-1}\big( e^{\gamma(\tau - \tau_0)} - 1 \big),$
		\item 
		\begin{equation*}
			\frac{2\beta}{N_s^3}\int_{0}^{\infty} y \tilde{T}(\tau,r)\partial_y(yF_s(y)) \d y 
			= -\frac{1 - \gamma}{\gamma} 
			\big(
			1 
			- (1 - 2\gamma) e^{-\gamma(\tau - r)}
			+ \Delta_1(r)
			\big)
		\end{equation*}
		for $|\Delta_1(r)| \leq c(\gamma, \eta)c_j.$ 
	\end{enumerate}
\end{corollary}

With this, we may rewrite (\ref{eqn-NG-duhamel-hyperbolic-selfsim}) to 
\begin{align*}
	N_G(\tau) =  - \frac{(1 - \gamma) }{\gamma} \int_{\tau_0}^{\tau}   \big( 1  -  (1 - 2\gamma)e^{-\gamma(\tau - r)} + \Delta_1(r) \big)\big( N_G(r) + H(r) \big)  \d r
	+ \int_{0}^{\infty} y \tilde{T}(\tau, \tau_0)G_0(y) \d y.
\end{align*}
Under suitable smallness assumptions on $\Delta_1, H$ and $\int_{0}^{\infty} y \tilde{T}(\tau, \tau_0)G_0(y) \d y$, this can be viewed as a perturbation of the integral equation
\begin{equation*}
	N(\tau) =  - \frac{(1 - \gamma) }{\gamma} \int_{\tau_0}^{\tau}   \big( 1  -  (1 - 2\gamma)e^{-\gamma(\tau - r)} \big)N(r)  \d r
\end{equation*}
The next lemma uses this idea to show that when $\Delta_1$ is small enough compared to $1$ and $H$ and the error term coming from $G_0$ are exponentially decaying, then $N_G$ is exponentially decaying as well:
\begin{lemma}\label{lemma-ODE-in-selfsim}
	Let $\tau_0 \geq 1$, $ \eta \in (0, 1 - \gamma)$ and let $\Delta_1, \Delta_2, H \in C([\tau_0, \infty))$ be such that 
	\begin{align}
		|\Delta_1(\tau)| & \leq \frac{\gamma^3(1 - \gamma - \eta)}{16}  \label{lemma-ODE-in-selfsim_c1} \\
		|\Delta_2(\tau)| &\leq c_2 e^{-\eta (\tau - \tau_0)} \nonumber \\
		|H(\tau)| &\leq c_{H}e^{-\eta (\tau - \tau_0)} \nonumber
	\end{align}
	for $c_H, c_2 > 0$, and let $N(\tau) \in C([\tau_0, \infty))$ satisfy
	\begin{equation}
		N(\tau) =  - \frac{(1 - \gamma) }{\gamma} \int_{\tau_0}^{\tau}   \big( 1 + \Delta_1(r) -  (1 - 2\gamma)e^{-\gamma(\tau - r)} \big)\big( N(r) + H(r) \big)  \d r
		+ \Delta_2(\tau).
	\end{equation}
	Then there exists $c(\gamma, \eta) > 0$ such that
	\begin{equation*}\label{ODE-lemma-exp-decay}
		|N(\tau)| \leq c(\gamma, \eta)(c_2 + c_H)e^{-\eta (\tau - \tau_0)}
	\end{equation*}
	for all $\tau \geq \tau_0.$
\end{lemma}

\begin{proof}
	Let $X(\tau) = \int_{\tau_0}^{\tau} (1 + \Delta_1(r))\big(N(r) + H(r) \big) \d r, Y(\tau) = e^{-\gamma \tau}\int_{\tau_0}^{\tau}  e^{\gamma r}\big(N(r) + H(r) \big) \d r.$ Then
	\begin{equation}\label{eqn-ode-lemma-xyprime}
	\begin{split}
		\begin{pmatrix}
			X' \\
			Y'
		\end{pmatrix}
		&= \begin{pmatrix}
			-\frac{(1-\gamma)}{\gamma} & \frac{(1-\gamma)(1 - 2\gamma)}{\gamma} \\
			-\frac{(1-\gamma)}{\gamma} & \frac{(1-\gamma)(1 - 2\gamma)}{\gamma} - \gamma
		\end{pmatrix}
		\begin{pmatrix}
			X \\
			Y
		\end{pmatrix} \\
		&\quad + 	\begin{pmatrix}
			(\Delta_2 + H) (1  + \Delta_1)   -\frac{(1-\gamma)}{\gamma}\Delta_1(X - (1 - 2\gamma)Y) \\
			\Delta_2 + H
		\end{pmatrix}.
	\end{split}
	\end{equation}
Since $X(\tau_0) = Y(\tau_0) = 0,$ it holds that $|X(\tau)|, |Y(\tau)| \leq c_X e^{-\eta(\tau - \tau_0)}$ on some time interval $[\tau_0, \tau^*]$ for
\begin{equation}\label{eqn-cx-bound}
	\frac{c_X(1 - \gamma - \eta)\gamma^2}{8}  \coloneqq 3(c_2  + c_H).
\end{equation}
Using this and (\ref{lemma-ODE-in-selfsim_c1}) we find
\begin{align*}
	&\quad \; \Big|(\Delta_2 + H) (1  + \Delta_1)   -\frac{(1-\gamma)}{\gamma}\Delta_1(X -(1 - 2\gamma)Y)\Big| (r) 
	+ \Big|\Delta_2 + H \Big|(r)
	\\
	&\leq 3\big(c_2e^{-\eta (r - \tau_0)} + c_He^{- \eta (r - \tau_0)} \big) + \frac{2}{\gamma} \frac{\gamma^3(1 - \gamma - \eta)}{16} c_X e^{-\eta(r - \tau_0)} \\ 
	&\leq \frac{c_X}{4} \gamma^2 (1 - \gamma - \eta) e^{-\eta (s - \tau_0)}.
\end{align*}

	The solutions of the homogeneous problem starting at time $s \leq \tau$ with datum $\begin{pmatrix}
		X_0 \\ Y_0
	\end{pmatrix}$
	are of the form
	\begin{equation*}
		\begin{pmatrix}
			X_{hom} \\ Y _{hom}
		\end{pmatrix}(\tau; s)
		= ae^{-(1-\gamma)(\tau - s)}
		\begin{pmatrix}
			1-2\gamma  \\
			1-\gamma
		\end{pmatrix}
		+ be^{-(\tau - s)}
		\begin{pmatrix}
			1 - \gamma \\
			1
		\end{pmatrix}
	\end{equation*}
	where $a = -\gamma^{-2}( X_0 + (1 - \gamma)Y_0)$ and $b = \gamma^{-2}( (1 - \gamma)X_0 -(1 - 2\gamma)Y_0).$ Thus using $\gamma \in \big(0, \frac{1}{2}\big)$ we obtain
	\begin{equation*}
		|X_{hom}(\tau, s)| \leq 2\gamma^{-2}\big(|X_0| + |Y_0|\big)e^{-(1-\gamma)(\tau - s)},
	\end{equation*}
	and the same for $|Y_{hom}|.$
	Hence Duhamel's formula for the solutions of the inhomogeneous problem with initial datum zero yields
	\begin{align*}
		|X(\tau, \tau_0)|
		&\leq 2\gamma^{-2} \int_{\tau_0}^{\tau} \frac{c_X}{4} \gamma^2(1 - \gamma - \eta)
		e^{-\eta (s - \tau_0)}e^{-(1-\gamma)(\tau - s)}
		\d s 
		= \frac{c_X}{2}e^{-\eta (\tau - \tau_0)},
	\end{align*}
	so that $|X|$ is indeed decaying exponentially with rate $\eta$ and a strictly smaller multiplicative constant than initially assumed. The same computation holds for $|Y|$, so in both cases $\tau^*$ can be extended to infinity and we obtain exponential decay for all times. 
	Using (\ref{eqn-cx-bound}) it follows for $N$ that
	\begin{align*}
		|N(\tau)| 
		&= \Big|- \frac{1 - \gamma}{\gamma} X(\tau) + \frac{(1 - \gamma)(1 - 2\gamma)}{\gamma} Y(\tau) + \Delta_2(\tau)\Big| \\
		&\leq \frac{c_X}{\gamma}e^{-\eta (\tau - \tau_0)} + c_2e^{-\eta (\tau - \tau_0)} 
		\leq  \frac{2c_X}{\gamma}e^{-\eta (\tau - \tau_0)}.
	\end{align*}
\end{proof}
We remark again that while $c_2$ and $c_H$ can be any positive constant, the bound for $\Delta_1$ depends on $\gamma$ and $\eta$ and is potentially very small. This is because $\Delta_1$ is multiplied with $X$ and $Y$ in the inhomogeneity (\ref{eqn-ode-lemma-xyprime}). On the other hand, it is not necessary to have exponential decay of $\Delta_1$ since $\Delta_2$ and $H$ are decaying exponentially by assumption, and $X$ and $Y$ by assumption of the continuation argument.

Finally, we identify the pointwise limit of $\tilde{T}(\tau, \tau_0)\Phi(y)$ as $\tau \to \infty$ when $J$ has exponential decay:
\begin{lemma}\label{lemma-F-to-Fs}
	Let $\tau_0, \eta > 0$ and let $J$ as in (\ref{eqn-def-J}) satisfy $|J(\tau)| \leq c_J e^{-\eta(\tau - \tau_0)}$ for $c_J, \eta > 0$ and $\tau \geq \tau_0$. Let $\Phi \in L^1((1+y)\d y)$ satisfy 
	$$\lim\limits_{y \to \infty} \Phi(y)y^{\theta} = c_\Phi \gamma^{-\theta}$$ for some $c_\Phi > 0$ and
	\begin{align}\label{eqn-F-to-Fs-assumption}
		\Big|\int_{0}^{\infty} y \tilde{T}(\tau, \tau_0)\Phi(y)  \d y - N_s \Big| 
		\leq c_N e^{-\eta (\tau - \tau_0)}
	\end{align}
	for some $c_N < N_s$. Then $\tilde{T}(\tau, \tau_0)\Phi(y) \to F_s(y)$ as $\tau \to \infty$ uniformly on compact sets in $[0, \infty)$.
\end{lemma}

\begin{proof}
	We adopt the strategy of \cite{Carr_asymptotic_1998}, who treat a similar transport equation with a constant first moment. We replace this by assuming the first moment converges exponentially to $N_s$, which is enough to obtain the same pointwise limit.
	
	By abuse of notation we denote $N_\Phi(\tau) = \int_{0}^{\infty}\tilde{T}(\tau, \tau_0)\Phi(y) y \d y$ and observe that $N_\Phi > 0$ by (\ref{eqn-F-to-Fs-assumption}). Integrating (\ref{eqn-def-tildeTJ}) gives
	\begin{equation}\label{eqn-carr-penrose-type}
		N_\Phi(\tau) = \frac{1}{M'(\tau)}e^{\frac{1}{2}(\tau - \tau_0)} \int_{M(\tau)}^{\infty} (y - M(\tau)) \Phi_0(y) \d y.
	\end{equation}
	We denote
	\begin{equation*}
		\nu(r) \coloneqq \int_{r}^{\infty} (y - r)\Phi(y) \d y
	\end{equation*}
	and observe that 
	\begin{equation*}
		\lim\limits_{r \to \infty} r^{\theta - 2}\nu (r)
		= c_\Phi \frac{\gamma^{-\theta}}{(\theta - 1)(\theta - 2)}.
	\end{equation*} 
	
	Hence (\ref{eqn-carr-penrose-type}) can be rewritten to
	$M'(\tau) N_\Phi(\tau) = e^{\frac{1}{2}(\tau - \tau_0)} \nu(M(\tau))$, which yields
	\begin{align*}
		\int_{\tau_0}^{M(\tau)} \frac{1}{\nu(r)} \d r 
		&= \int_{\tau_0}^{\tau} \frac{e^{\frac{1}{2} (r - \tau_0)}}{N_\Phi(r)} \d r \\
		&= \int_{\tau_0}^{\tau} \frac{e^{\frac{1}{2} (r - \tau_0)}}{N_s} \d r 
		- \int_{\tau_0}^{\tau} e^{\frac{1}{2} (r - \tau_0)} \frac{N_\Phi - N_s}{N_s N_\Phi} \d r \\
		&= \frac{2}{N_s} \big(e^{\frac{1}{2}(\tau - \tau_0)} - 1\big)
		- \int_{\tau_0}^{\tau} e^{\frac{1}{2} (r - \tau_0)} \frac{N_\Phi - N_s}{N_s N_\Phi} \d r.
	\end{align*}
	with
	\begin{equation*}
		\Big| \int_{\tau_0}^{\tau} e^{\frac{1}{2} (r - \tau_0)} \frac{N_\Phi - N_s}{N_s N_\Phi} \d r \Big|
		\leq 
		\left\{
		\begin{split}
			&c(\gamma, c_N, \eta)e^{(\frac{1}{2} - \eta)(\tau - \tau_0)}, \quad &\eta \neq \frac{1}{2}, \\
			&c(\gamma, c_N)(\tau - \tau_0), &\eta = \frac{1}{2}
		\end{split}
		\right.
	\end{equation*}
	as $\tau \to \infty$. 
	\Cref{lemma-hyperbolic-M} implies $M(\tau) \to \infty$ for $\tau \to \infty$ and more specifically 
	$$\liminf\limits_{\tau \to \infty} e^{-\gamma (\tau - \tau_0)}M(\tau) > 0.$$
	Because of this, using L'Hôpital's rule and the fact that $\gamma(\theta - 1) = \frac{1}{2}$ it holds that
	\begin{align*}
		\lim\limits_{\tau \to \infty} \frac{e^{\frac{1}{2} (\tau - \tau_0)}}{M(\tau)^{\theta - 1}}
		&= \lim\limits_{\tau \to \infty} \frac{1}{M(\tau)^{\theta - 1}}
		\Big(
		\frac{N_s}{2} \int_{\tau_0}^{M(\tau)} \frac{1}{\nu(r)} \d r 
		+ \frac{N_s}{2} \int_{\tau_0}^{\tau} e^{\frac{1}{2} (r - \tau_0)} \frac{N_\Phi - N_s}{N_s N_\Phi} \d r
		+ 1
		\Big) \\
		&= \lim\limits_{\tau \to \infty} \frac{N_s}{2}
		\frac{1}{M(\tau)^{\theta - 1}}
		\int_{\tau_0}^{M(\tau)} \frac{1}{\nu(r)} \d r 
		= \lim\limits_{\tau \to \infty} \frac{N_s}{2} \frac{1}{(\theta - 1) \nu(M(\tau)) M(\tau)^{\theta - 2}} \\
		&= \frac{N_s(\theta - 2)\gamma^{\theta}}{2 c_\Phi} .
	\end{align*}
	Next, we observe that
	\begin{align*}
		\lim\limits_{\tau \to \infty} \frac{M'(\tau)}{M(\tau)}
		&= \lim\limits_{\tau \to \infty} \frac{e^{\frac{1}{2}(\tau - \tau_0)}}{M(\tau)^{\theta - 1}} \frac{ 1}{N_\Phi(\tau)} \nu(M(\tau)) M(\tau)^{\theta - 2}
		= \frac{N_s(\theta - 2)\gamma^{\theta}}{2 c_\Phi} \frac{c_\Phi \frac{\gamma^{-\theta}}{(\theta - 1)(\theta - 2)}}{N_s} 
		= \frac{1}{2(\theta - 1)} = \gamma
	\end{align*}
	and therefore $\lim\limits_{\tau \to \infty} \frac{M'(\tau)y + M(\tau)}{M(\tau)} = 1 + \gamma y$ for all $y \geq 0$. 
	All together, using $c_s = \frac{1 - 2\gamma}{4}N_s = \frac{(\theta - 2)\gamma}{2} N_s$ we obtain for all $y \geq 0$
	\begin{align*}
		&\quad \lim\limits_{\tau \to \infty} \tilde{T}(\tau, \tau_0)\Phi(y) \\
		&= \lim\limits_{\tau \to \infty} e^{\frac{1}{2}(\tau - \tau_0)}M'(\tau)\Phi(M'(\tau)y + M(\tau)) \\
		&= \lim\limits_{\tau \to \infty}
		\frac{e^{\frac{1}{2}(\tau - \tau_0)}}{M(\tau)^{\theta - 1}}
		\Big( \frac{M'(\tau)y + M(\tau)}{M(\tau)} \Big)^{-\theta}
		\frac{M'(\tau)}{M(\tau)}
		\big( M'(\tau)y + M(\tau) \big)^{\theta} \Phi(M'(\tau) y + M(\tau)) \\
		&= \frac{N_s(\theta - 2)\gamma^{\theta}}{2 c_\Phi}
		\big( 1 + \gamma y \big)^{-\theta} \gamma c_\Phi \gamma^{-\theta}
		= c_s\big( 1 + \gamma y \big)^{-\theta} = F_s(y),
	\end{align*}
	and thus by Dini's theorem uniform convergence on compact sets. 
\end{proof}

\subsection{Proof of \Cref{thm-main-hyperbolic}}
Combining the previous results now allows us to prove \Cref{thm-main-hyperbolic}.

\begin{proof}
	$N_G$ satisfies (\ref{eqn-NG-duhamel-hyperbolic-selfsim}) with $J, H$ and $\tilde{T}$ as defined in $(\ref{eqn-def-J}), (\ref{eqn-def-H})$ and $(\ref{eqn-def-tildeTJ})$. We will use this integral equation to deduce exponential decay of $N_G(\tau)$ from the smallness assumption $(\ref{thm-main-hyperbolic-ass-G0})$. Indeed by $(\ref{thm-main-hyperbolic-ass-G0})$ it holds
	\begin{equation*}
		|N_G(\tau_0)| \leq c_0 \gamma^{-2} \Big( \frac{1}{\theta + \varepsilon - 2} - \frac{1}{\theta + \varepsilon - 1}  \Big) + 4e^{-\varepsilon \gamma \tau_0}
		\leq \frac{c_0}{\gamma^2(\theta - 2)( \theta - 1)} + 4e^{-\varepsilon \gamma \tau_0}
	\end{equation*}
	and therefore
	\begin{equation}\label{eqn-hyperbolic-ng-assumption}
		|N_G(\tau_0)| \leq c_N e^{-\varepsilon \gamma (\tau - \tau_0)}
	\end{equation}
	on a maximal interval $[\tau_0, \tau^*)$ for
	\begin{equation}\label{eqn-hyperbolic-cg-lowerbound}
		c_N \geq  \frac{2c_0}{\gamma^2(\theta - 2)( \theta - 1)} + 8e^{-\varepsilon \gamma \tau_0}.
	\end{equation}
	 We will show that under some additional assumptions on $c_N$, this implies $|N_G(\tau_0)| \leq \frac{c_N}{2} e^{-\varepsilon \gamma (\tau - \tau_0)}$ for all $\tau \in [\tau_0, \tau^*)$, so that $(\ref{eqn-hyperbolic-ng-assumption})$ can be extended to $\tau \in [\tau_0, \infty)$. 

	 \textbf{Step 1:} Smallness of $J$
	 
	 We assume $c_0$ is small and $\tau_0$ large enough so that we may take $c_N$ small enough such that
	 \begin{equation}\label{eqn-hyperbolic-cg-upperbound}
	 	c_N \leq \frac{N_s}{4}
	 \end{equation}
	 without violating (\ref{eqn-hyperbolic-cg-lowerbound}).
	 This implies that $N_s^2 + 2N_s N_G  \geq \frac{N_s^2}{2}$, 
	 so that
	 $$N_s^4 \big(1 + \frac{e^{-\tau} + 2N_sN_G + N_G^2}{N_s^2} \big)
	 \geq N_s^2 \big(N_s^2 + 2N_G N_s\big)
	 = \frac{N_s^4}{2}$$ 
	 and therefore using $\varepsilon \gamma \in (0, 1]$ and (\ref{eqn-hyperbolic-cg-upperbound})
	 \begin{align}\label{eqn-hyperbolic-H-small}
	 	|H(N_G, \tau)| &\leq c(\gamma)\big(e^{-\tau} + c_N^2 e^{-2\varepsilon \gamma (\tau - \tau_0)} + c_N e^{-\varepsilon \gamma (\tau - \tau_0)}\big(e^{-\tau} + c_N e^{-\varepsilon \gamma (\tau - \tau_0)} + c_N^2 e^{-2\varepsilon \gamma (\tau - \tau_0)} \big) \big) \nonumber \\
	 	&\leq c(\gamma)(e^{-\varepsilon \gamma \tau_0} + c_N^2)e^{-\varepsilon \gamma (\tau - \tau_0)}
	 \end{align}
	 and
	 \begin{align}\label{eqn-hyperbolic-J-small}
	 	\big|J(\tau) \big|
	 	=
	 	\big|\frac{2\beta N_G(\tau) }{N_s^3}
	 	+ \beta H(N_G, \tau)
	 	\big| 
	 	\leq c(\gamma)(e^{-\varepsilon \gamma \tau_0} + c_N + c_N^2)e^{-\varepsilon \gamma (\tau - \tau_0)}.
	 \end{align}
	 
	 \textbf{Step 2}: Estimation of the terms in (\ref{eqn-NG-duhamel-hyperbolic-selfsim}).
	 
	 This bound on $J$ allows us to apply \Cref{lemma-hyperbolic-M} with $\eta = \varepsilon \gamma \in (0, 1 - \gamma) \subset (0, 1)$ to estimate the terms in $(\ref{eqn-NG-duhamel-hyperbolic-selfsim}).$ 
	 Firstly, we can bound
	 \begin{align*}
	 	\Big| \int_{0}^{\infty} x\tilde{T}(\tau,\tau_0)G_0 \d r\Big|
	 	&= \Big| \frac{e^{\frac{1}{2} (\tau - \tau_0)}}{M'(\tau, \tau_0)} \int_{M(\tau, \tau_0)}^{\infty} (y - M(\tau, \tau_0))G_0(y) \d y \Big| \\
	 	&\leq c_0 c(\gamma, \varepsilon)  e^{( \frac{1}{2}- \gamma)(\tau - \tau_0)} 
	 	\int_{M(\tau, \tau_0)}^{\infty} (y - M(\tau, \tau_0))\big(1 + \gamma y \big)^{-(\theta +\varepsilon)} \d y
	 	\\
	 	&= c_0 c(\gamma, \varepsilon) e^{( \frac{1}{2}- \gamma)(\tau - \tau_0)}  
	 	\Big(-
	 	\frac{1}{\theta + \varepsilon - 1}(1 + \gamma M(\tau, \tau_0))^{ -(\theta + \varepsilon -1)}(1 + \gamma M(\tau, \tau_0))  \\
	 	&\hspace*{4cm} + \frac{1}{ \theta + \varepsilon -2}(1 + \gamma M(\tau, \tau_0))^{-(\theta + \varepsilon -2)}
	 	\Big)
	 	\\
	 	&\leq c_0 c(\gamma,  \varepsilon)e^{( \frac{1}{2}- \gamma)(\tau - \tau_0)}  
	 	(1 +  \gamma M(\tau, \tau_0))^{-(\theta + \varepsilon -2)}
	 	\\
	 	&\leq c_0 c(\gamma, \varepsilon) e^{(\frac{1}{2}- \gamma + \gamma(2 - \theta - \varepsilon))(\tau - \tau_0)}  
	 	= c_0 c(\gamma,  \varepsilon) e^{- \varepsilon \gamma(\tau - \tau_0)}. 
	 \end{align*}
	Similarly, it holds
	\begin{align*}
		\Big| \int_{0}^{\infty} x\tilde{T}(\tau,\tau_0)\hat{G}_0 \d r \Big|
		&\leq \frac{e^{\frac{1}{2} (\tau - \tau_0)}}{M'(\tau, \tau_0)} e^{(2 - \varepsilon \gamma)\tau_0} 
		\Big|  \int_{M(\tau, \tau_0)}^{\infty} (y - M(\tau, \tau_0))e^{ \frac{ye^{\tau}}{2} } \d y \Big| \\
		&\leq c(\gamma) \frac{e^{\frac{1}{2} (\tau - \tau_0)}}{M'(\tau, \tau_0)} e^{(2 - \varepsilon \gamma)\tau_0} e^{-2\tau - \frac{M(\tau, \tau_0)e^{\tau}}{2}} 
		\leq c(\gamma)e^{-(\frac{3}{2} + \gamma )(\tau - \tau_0)} e^{-\varepsilon \gamma \tau_0}.
	\end{align*} 

Then with \Cref{lemma-hyperbolic-M}(iii) we can rewrite (\ref{eqn-NG-duhamel-hyperbolic-selfsim}) to
\begin{align*}
 N_G(\tau) 
	&= - \frac{1 - \gamma}{\gamma} \int_{\tau_0}^{\tau} 
	\Big( 
	N_G(r)+ \frac{N_s^3}{2\beta}H(r)
	\Big) 
	\Big(
	1 
	- (1 - 2 \gamma) \Big(\frac{t}{r} \Big)^{-\gamma}
	+ \Delta_1(r)
	\Big)
	\d r
	+ \Delta_2(t),
\end{align*}

where by $(\ref{eqn-hyperbolic-H-small}), (\ref{eqn-hyperbolic-J-small})$, \Cref{lemma-hyperbolic-M}$(iii)$ and the previous computation we have
\begin{align*}
	\Big|\frac{N_s^3}{2\beta}H(N_G, r)\Big| &\leq c(\gamma)(e^{-\varepsilon \gamma \tau_0} + c_N^2)e^{-\varepsilon \gamma (\tau - \tau_0)}
	 \\
	|\Delta_1(t)| &\leq c(\gamma, \varepsilon)(e^{-\tau_0} + c_N + c_N^2) \\
	|\Delta_2(t)| &=  	\Big| \int_{0}^{\infty} x\tilde{T}(\tau,\tau_0)(G_0 + \hat{G}_0) \d r \Big| 
	\leq 
	c(\gamma,  \varepsilon)(c_0 + e^{-\varepsilon \gamma \tau_0}) e^{- \varepsilon \gamma(\tau - \tau_0)}.
\end{align*}

\textbf{Step 3: } Conclusion

When $c_0$ is small and $\tau_0$ large enough, we can take $c_N$ sufficiently small that the bound $(\ref{lemma-ODE-in-selfsim_c1})$ on $\Delta_1$ holds, so that we can apply \Cref{lemma-ODE-in-selfsim}. We obtain
\begin{align}\label{eqn-hyperbolic_NG-conclusion}
 	| N_G(\tau)|
	&\leq c(\gamma, \varepsilon)\big(e^{- \varepsilon \gamma \tau_0} + c_N^2 + c_0  \big)e^{-  \varepsilon \gamma (\tau - \tau_0)}.
\end{align}
We assume $c_0$ is small and $\tau_0$ large enough such that
$c(\gamma, \varepsilon)c_N^2 \leq \frac{c_N}{4}$
without violating (\ref{eqn-hyperbolic-cg-lowerbound}), and that $\tau_0$ is large enough such that
\begin{equation}\label{eqn-hyperbolic-cg-lowerbound2}
	c(\gamma,\varepsilon) \big( c_0 + e^{-\varepsilon \gamma \tau_0} \big) \leq \frac{c_N}{4}
\end{equation}
without violating (\ref{eqn-hyperbolic-cg-upperbound}) (note this is, up to constants, the same assumption as (\ref{eqn-hyperbolic-cg-lowerbound})).
Then (\ref{eqn-hyperbolic_NG-conclusion}) reads
$N_G(\tau) \leq \frac{c_N}{2} e^{-  \varepsilon \gamma (\tau - \tau_0)},$
so that the assumption (\ref{eqn-hyperbolic-ng-assumption}) can be extended from $[\tau_0, \tau^*)$ to $[\tau_0, \infty)$. The optimal choice of $c_N$ is determined by the lower bounds (\ref{eqn-hyperbolic-cg-lowerbound}) and (\ref{eqn-hyperbolic-cg-lowerbound2}).

Having shown $|N_G(\tau)| = |N_F(\tau) - N_s| \leq c_Ne^{-  \varepsilon \gamma (\tau - \tau_0)}$, $c_N < N_s$ for all $\tau \geq \tau_0$, we can now apply \Cref{lemma-F-to-Fs} with $\Phi = F_0 = F_s + G_0$ satisfying $\lim\limits_{y \to \infty} y^\theta F_0(y) = c_s\gamma^{-\theta}$ and conclude that
$$F(y, \tau) = \tilde{T}(\tau, \tau_0)F_0(y) \to F_s(y)$$
as $\tau \to \infty$ for all $y \geq 0$. 

\end{proof}

	\section{The parabolic problem}
	
	We now consider the full problem (\ref{eqn-main-parabolic}) including the second order term. 
	As in the hyperbolic case, we want to write $\frac{1}{1 + n_f^2}$ as a perturbation of $\frac{1}{n_s^2}$. 
	In this way, in Section \ref{sec-para-derivation-ng} we will obtain an integral equation for $N_F - N_s$ in terms of the solution of a suitable homogeneous parabolic problem, namely the parabolic variant of (\ref{eqn-pde-for-T}) with boundary value $0$, since for this operator the solution is linear in the initial data. Unlike the hyperbolic case, the parabolic problem cannot be solved by characteristics. Instead, in Section \ref{sec-parabolic-approxlemma} we show via maximum principle that the solutions of the homogeneous parabolic and hyperbolic problem are comparable when the initial data are sufficiently small. To use the maximum principle, up to this step we work in the original variables $x$ and $t$. 
	
	The statement of Theorem \ref{thm-main-parabolic} is in terms of $G = F - F_s$, however, we will work in the proof with 
	\begin{align*}
		G(y,\tau) &= F - F_s - \mu e^{\frac{3}{2}\tau - ye^{\tau}} - c_s e^{-ye^{\tau}}, \\
		g(x,t) &= f - f_s - \mu e^{-x} - \frac{c_s}{t^{\frac{3}{2}}}e^{-x} 
		= t^{-\frac{3}{2}}G(y, \tau). 
	\end{align*}
	This is to remove the boundary layer at $x = 0$ resp. $y = 0$, i.e. to have $g(0, t) = G(0, \tau) = 0$ for all $t \geq t_0$ resp. $\tau \geq \tau_0$. We remark that the assumptions on $G_0, \hat{G}_0$ in \Cref{thm-main-parabolic} carry over to this definition of $G$. 
	We denote
	$$f_{bl}(x,t) = \mu e^{-x} + \frac{c_s}{t^{\frac{3}{2}}}e^{-x} = t^{-\frac{3}{2}}F_{bl}(y, \tau), \quad
	n_{bl}(t) = \mu - \frac{c_s}{t^{\frac{3}{2}}},$$
	assume $ \bar{t}_0 > \Big(\frac{c_s}{\mu} \Big)^{\frac{2}{3}}$ and observe that then $n_{bl}(t) > 0$ for all $t \geq \bar{t}_0.$
	
	\subsection{Derivation of the integral equation for $n_g$}\label{sec-para-derivation-ng}
	
	The following variant of \Cref{prop-ng} holds:
	\begin{proposition}\label{prop-ng-parabolic}
		Let $t_0 > 1$ be large enough such that $n_{bl}(t) \geq 0$ for all $t \geq t_0$. Let \linebreak $g(x,t_0) = g_0(x) + \hat{g}_0(x).$
		Define 
		\begin{align}
			\label{eqn-def-h-para}
			h(n_g, t) &\coloneqq \frac{1 + n_g^2 + n_{bl}^2 +  2n_{bl} (n_g+n_s) 
				- 2n_g n_s^{-1}\big(1 + n_f^2 - n_{s}^2  \big)
			}
			{n_s^4 \big(1 + \frac{1 + n_f^2 - n_{s}^2 }{n_s^2}\big)}, \\
			\label{eqn-def-j-para}
			j(t) &\coloneqq \frac{2\beta n_g }{n_s^3}
			+ \beta h(n_g, t).
		\end{align}
		Let $S(t, t_0)\varphi_0$ be the solution mapping of
			\begin{equation}\label{approxlemma-parabolic-eqn}
			\left\{ 
			\begin{split}
				\partial_t \varphi &= \partial_x \Big( \partial_x \varphi + \Big(1  +\beta x\big( -\frac{1}{n_s^2} + j \big) \Big)\varphi   \Big), \quad x > 0 \\
				\varphi(0, t) &= 0 \\
				\varphi(x, 0) &= \varphi_0.
			\end{split}
			\right.
		\end{equation}  
		Then
		\begin{equation}\label{eqn-g-duhamel-parabolic}
			\begin{split}
				g(x, t) &= S(t, t_0)(g_0 + \hat{g}_0)
				+ \int_{t_0}^{t} 
				\beta \Big(\frac{2n_g(r) }{n_s^3(r)}
				+ h(n_g(r), r)\Big)  S(t, r)  \partial_x(xf_s(x,r)) 
				\d r \\
				&\quad+ \int_{t_0}^{t} S(t,r)\big(\mathcal{R}_{bl}(x,r) + \partial_x^2 f_s(x,r) \big) \d r
			\end{split}
		\end{equation}
		and
		\begin{equation}\label{eqn-ng-duhamel-parabolic}
			\begin{split}
				n_g(t) &= \int_{0}^{\infty} x S(t, t_0)(g_0 + \hat{g}_0) \d x \\
				&\quad+ \int_{t_0}^{t} 
				\beta\Big(\frac{2n_g(r) }{n_s^3(r)}
				+ h(n_g(r), r)\Big) \int_{0}^{\infty} x S(t, r)  \partial_x(xf_s(x,r)) \d x
				\d r \\
				&\quad+ \int_{t_0}^{t} \int_{0}^{\infty} x S(t,r)\big( \mathcal{R}_{bl}(x,r) + \partial_x^2 f_s(x,r) \big) \d x \d r
			\end{split}
		\end{equation}
		for
		\begin{align}\label{eqn-def-Rbl}
			\mathcal{R}_{bl}(x,t) 
			&= \beta \big(\frac{1}{n_s^2} - j \big)\big(\mu + \frac{c_s}{t^{\frac{3}{2}}} \big)\partial_x(xe^{-x}) - \frac{3}{2}c_st^{-\frac{5}{2}}e^{-x}.
		\end{align}
	\end{proposition}
	
	\begin{proof}
		We have $n_f = n_s + n_g + n_{bl}$, so that
		\begin{align*}
			\frac{1}{n_s^2} - \frac{1}{ n_f^2}
			&= \frac{1 +  n_f^2 - n_s^2}{n_s^2(1 +  n_f^2)}
			= \frac{1 + n_g^2 + n_{bl}^2 + 2n_g n_s + 2n_{bl} n_g  + 2n_{bl} n_s }{n_s^4 \Big(1 + \frac{1 + n_f^2 - n_{s}^2 }{n_s^2}\Big)} \\
			&= \frac{2n_g}{n_s^3} + h(n_g, t)
		\end{align*}
		with $h$ defined as above. One easily checks that
		\begin{equation}\label{eqn-diffeq-g-parabolic}
			\left\{
			\begin{split}
				\partial_t g 
				&= \partial_x^2 g + \Big(1 + \beta x \Big(-\frac{1}{n_s^2} + j(t) \Big)  \Big)\partial_x g
				+ \beta \Big(-\frac{1}{n_s^2} + j(t) \Big)g \\
				&\quad + \beta j(t)
				\partial_x^2 f_s + \partial_x(xf_s)
				+ \mathcal{R}_{bl}(x,t), \\
				g(0, t) &= 0, \\
				g(x, t_0) &= g_0(x).
			\end{split}
			\right.
		\end{equation}
		This yields (\ref{eqn-g-duhamel-parabolic}) by Duhamel's formula and (\ref{eqn-ng-duhamel-parabolic}) by integration. 
	\end{proof}

	\subsection{Approximation lemma}\label{sec-parabolic-approxlemma}
	
	To understand (\ref{eqn-ng-duhamel-parabolic}) we must understand the operator $S(t,t_0)$. The following approximation lemma relates $S(t,t_0)$ to $T(t,t_0),$ the solution mapping of the hyperbolic problem.
	We work in the original variables $(x,t)$ so that we can use the maximum principle.

	\begin{lemma}\label{approx-lemma}
		Let $\lambda \in \big[\frac{3}{2}, \frac{3}{2} + 2\gamma\big)$ and $j$ be as in (\ref{eqn-def-j-para}). Then there exist $\bar{t}_0 \geq 1, c_{1}, c_{2} > 0$ such that if
		\begin{equation}\label{approx-lemma-max-prpl-cond}
			| j(t) | \leq c_j t_0^\eta t^{-1 - \eta}, 
			\quad c_j \leq
			\min \Big\{\frac{1}{2N_s^2},  \frac{\frac{3}{2} + 2\gamma - \lambda }{8\beta (\theta + 2)} , \frac{\gamma}{5\theta} \Big\},
		\end{equation}
		for $c_j, \eta > 0$ and $t \geq t_0 \geq \bar{t}_0$, and $\varphi_0 \in L^1$ satisfies
		\begin{equation}\label{eqn-approxlemma-varphi-bound-f_s}
			|\varphi_0(x)| \leq   s^{-\frac{3}{2}}\Big(1 + \frac{\gamma x}{s}\Big)^{-\theta}, \quad
			|\varphi_0''(x)| \leq   s^{-\frac{7}{2}} \Big(1 + \frac{\gamma x}{s}\Big)^{-(\theta + 2)}
		\end{equation}
		for $s \geq t_0$, then
		\begin{equation}\label{eqn-lemma-sol-of-parabolic}
			\big| S(t, s)\varphi_0 - T(t,s)\varphi_0 \big|(x) 
			\leq 
			c_1t^{-\frac{3}{2}}e^{-\frac{x}{2}} + c_2t^{-\lambda} \Big(1 + \gamma \frac{x}{t} \Big)^{-\theta - 2}
			\quad \text{ for all } x \geq 0, t \geq s
		\end{equation}
		and
		\begin{equation}\label{eqn-lemma-parabolic-approx-integrated}
			\int_{0}^{\infty} x|S(t,s)\varphi_0 - T(t,s)\varphi_0| \d x
			\leq c(\gamma, \lambda) \big( t^{-\frac{3}{2}} + t^{2 - \lambda}\big).
		\end{equation}
		If instead of $(\ref{eqn-approxlemma-varphi-bound-f_s})$ it holds
		\begin{equation}\label{eqn-approxlemma-varphi-bound-exp}
			|\varphi_0(x)| \leq e^{-\frac{x}{2}},
		\end{equation}
		then
		\begin{equation}\label{eqn-lemma-sol-of-parabolic-exponential}
			| S(t,s)\varphi_0(x) |
			\leq 
			\Big(\frac{t}{s} \Big)^{-\frac{3}{2}}e^{-\frac{x}{2}}
			+ c_2t^{-2}\Big(\frac{t}{s} \Big)^{-\frac{3}{2}} \Big(1 + \gamma \frac{x}{t} \Big)^{-5\theta }
			\quad \text{ for all } x \geq 0, t \geq s
		\end{equation}
		and
		\begin{equation}\label{eqn-lemma-parabolic-approx-integrated-exponential}
			\int_{0}^{\infty} x|S(t,s)\varphi| \d x
			\leq c(\gamma)\Big(\frac{t}{s} \Big)^{-\frac{3}{2}}.
		\end{equation}
	\end{lemma}

	\begin{proof}
		By $(\ref{approx-lemma-max-prpl-cond})$ it holds for all $t \geq \bar{t}_0$
		$$\Big(\frac{\beta}{n_s^2(t)} - \beta j(t)\Big) \geq \frac{\beta}{2 N_s^2 t} \geq 0.$$
		Therefore the maximal principle applies to the operator
		\begin{equation*}
			L[\varphi] \coloneqq \partial_t \varphi - \partial_x^2 \varphi -  \Big(1 + \beta x \Big(-\frac{1}{n_s^2} + j(t) \Big)  \Big)\partial_x \varphi
			+ \beta \varphi \Big(\frac{1}{n_s^2} - j(t) \Big)
		\end{equation*}
		Hence in case (\ref{eqn-approxlemma-varphi-bound-f_s}), if we can show that
		\begin{enumerate}[(a)]
			\item $\big|L\big[S(t,s)\varphi_0 - T(t,s)\varphi_0 \big] \big|  \leq L\Big[ c_1t^{-\frac{3}{2}}e^{-\frac{x}{2}} + c_2t^{-\lambda}\Big(1 + \gamma \frac{x}{t} \Big)^{-\theta - 2} \Big],$
			\item $\big|S(t,s)\varphi_0(0) - T(t,s)\varphi_0(0) \big| = \big|T(t,s)\varphi_0(0) \big| \leq c_1t^{-\frac{3}{2}} + c_2 t^{-\lambda}$, and
			\item $\big|S(s,s)\varphi_0 - T(s,s)\varphi_0 \big| = 0 \leq c_1s^{-\frac{3}{2}}e^{-\frac{x}{2}} + c_2s^{- \lambda}\Big(1 + \gamma \frac{x}{s} \Big)^{-\theta - 2}$, which is trivially true,
		\end{enumerate}
		then (\ref{eqn-lemma-sol-of-parabolic}) follows by maximum principle and (\ref{eqn-lemma-parabolic-approx-integrated}) by integration, since then
		$$\int_{0}^{\infty} x|S(t,s)\varphi_0 - T(t,s)\varphi_0| \d x
		\leq 4c_1t^{-\frac{3}{2}} + \frac{c_2}{\gamma^2 \theta (\theta + 1)} t^{2 - \lambda}.$$
		
		In case (\ref{eqn-approxlemma-varphi-bound-exp}) we forgo $T(t,s)\varphi_0$ and directly approximate $S(t,s)\varphi_0$, such that it suffices to show
		\begin{enumerate}[(a')]
			\item $\big|L\big[S(t,s)\varphi_0 \big] \big| =0  
			\leq  L \Big[\Big(\frac{t}{s} \Big)^{-\frac{3}{2}}e^{-\frac{x}{2}} + c_2t^{-2}\Big(\frac{t}{s} \Big)^{-\frac{3}{2}} \Big(1 + \gamma \frac{x}{t} \Big)^{-5\theta} \Big],$
			\item $\big|S(t,s)\varphi_0(0) \big| = 0 \leq \Big(\frac{t}{s} \Big)^{-\frac{3}{2}} + c_2t^{-2}\Big(\frac{t}{s} \Big)^{-\frac{3}{2}}$, which always holds, and
			\item $\big|S(s,s)\varphi_0  \big| = |\varphi_0(x)| \leq e^{-\frac{x}{2}} + c_2s^{-2}\Big(1 + \gamma \frac{x}{s} \Big)^{-5\theta}$, which is trivially true. 
		\end{enumerate}
		We begin with showing $(a)$, starting with the exponential terms. We have
		\begin{align*}
			L[c_{1}t^{-\frac{3}{2}}e^{-\frac{x}{2}}] 
			&= c_1t^{-\frac{3}{2}}e^{-\frac{x}{2}} 
			\Big(
			- \frac{3}{2t} 
			- \frac{1}{4}
			+ \frac{1}{2}
			\Big(1 + \beta x \Big(-\frac{1}{n_s^2} + j \Big) \Big) 
			+  \beta \Big( \frac{1}{n_s^2} - j \Big)
			\Big) \\
			&\geq c_1t^{-\frac{3}{2}}e^{-\frac{x}{2}} 
			\Big(
			\frac{1}{8} 
			- \frac{3 \beta}{2 N_s^2 } \frac{x}{t} 
			\Big).
		\end{align*}
		whenever $\bar{t}_0 \geq 12$. Since $\frac{3 \beta}{2 N_s^2 } = \frac{3 }{2 }\big(1 - \gamma\big) \leq \frac{ 3}{2},$ for all $\frac{x}{t} \leq \frac{1}{12}$ it follows 
		$L[c_{1}t^{-\frac{3}{2}}e^{-\frac{x}{2}}] \geq 0,$
		whereas for all $\frac{x}{t} \geq \frac{1}{12}$ we have
		\begin{align*}
			\frac{x}{t}  e^{-\frac{x}{2}} 
			&= t^{-3} xt^2 e^{-\frac{x}{2}}
			\leq 12^2t^{-3} x^3 e^{-\frac{x}{2}}
			\leq c(\gamma)t^{-3} \Big(1 + \gamma \frac{x}{t} \Big)^{-\theta -2}
		\end{align*}
		and
		\begin{equation*}
			L[c_{1}t^{-\frac{3}{2}}e^{-\frac{x}{2}}] 
			\geq - c(\gamma) c_1 t^{-\frac{9}{2}}  \Big(1 + \gamma \frac{x}{t} \Big)^{-2\theta}
			\geq - c(\gamma) c_1 t^{-\frac{7}{2}}  \Big(1 + \gamma \frac{x}{t} \Big)^{-\theta - 2 }.
		\end{equation*}
		We now consider the contribution of the second part of the right hand side of $(a)$. We have
		\begin{align*}
			&\quad \; L
			\Big[
			c_{2}t^{-\lambda}\Big(1 + \gamma \frac{x}{t} \Big)^{-\theta - 2}
			\Big] \\
			&= c_{2}t^{-\lambda - 1}
			\Big(1 + \gamma \frac{x}{t} \Big)^{-\theta -2} 
			\Big( 
			- \lambda
			+ \frac{x}{t}\gamma(\theta + 2) \Big(1 + \gamma \frac{x}{t} \Big)^{-1}
			- (\theta + 2)(\theta + 3)\frac{\gamma^2}{t} \Big(1 + \gamma \frac{x}{t} \Big)^{-2} \\
			&\hspace*{4,7cm} + \Big(1 + \beta x\Big(-\frac{1}{n_s^2} + j \Big) \Big) \gamma (\theta + 2)\Big(1 + \gamma \frac{x}{t} \Big)^{-1}
			+ \beta  t \Big(\frac{1}{n_s^2} - j \Big)
			\Big) \\
			&= c_{2}t^{-\lambda - 1}
			\Big(1 + \gamma \frac{x}{t} \Big)^{-\theta -2} 
			\Big( 
			- \lambda
			+ \gamma(\theta + 2) \Big(1 + \gamma \frac{x}{t} \Big)^{-1} \Big(1 + \frac{x}{t}\Big(1 - \frac{\beta}{N_s^2}\Big)  \Big)
			\\
			&\hspace*{4,7cm} - (\theta + 2)(\theta + 3)\frac{\gamma^2}{t} \Big(1 + \gamma \frac{x}{t} \Big)^{-2} \\
			&\hspace*{4,7cm} + \beta t j(t) \Big(  \gamma \frac{x}{t} (\theta + 2)\Big(1 + \gamma \frac{x}{t} \Big)^{-1} - 1\Big)
			+ \frac{\beta}{N_s^2}
			\Big).  
		\end{align*}
		Using (\ref{eqn-def-gamma}), (\ref{approx-lemma-max-prpl-cond}) and the fact that $\Big|\gamma \frac{x}{t} (\theta + 2)\Big(1 + \gamma \frac{x}{t} \Big)^{-1} - 1  \Big| \leq \theta + 1$ we further estimate
		\begin{align*}
			&\quad \;L\Big[
			c_{2}t^{-\lambda}\Big(1 + \gamma \frac{x}{t} \Big)^{-\theta - 2}
			\Big] \\
			&\geq c_{2}t^{-\lambda - 1}
			\Big(1 + \gamma \frac{x}{t} \Big)^{-\theta -2} 
			\Big( 
			- \lambda
			+ \gamma(\theta + 2)
			- (\theta + 2)(\theta + 3)\frac{\gamma^2}{t} - (\theta + 1)c_j
			+ 1 - \gamma
			\Big)
			\\
			&\geq c_{2}t^{-\lambda - 1}
			\Big(1 + \gamma \frac{x}{t} \Big)^{-\theta -2} 
			\Big( 
			- \lambda
			+ 3 \gamma + \frac{1}{2}
			+ 1 - \gamma
			- (\theta + 2)(\theta + 3)\frac{\gamma^2}{\bar{t}_0} 
			- \beta (\theta + 1)c_j
			\Big) \\ 
			&\geq \frac{c_{2}}{2}\Big(2\gamma + \frac{3}{2} - \lambda \Big)
			t^{-\lambda - 1}
			\Big(1 + \gamma \frac{x}{t} \Big)^{-\theta -2} 
		\end{align*}
		whenever $\bar{t}_0$ is large enough such that $(\theta + 2)(\theta + 3)\gamma^2 \leq \frac{\bar{t}_0 (2\gamma + \frac{3}{2} - \lambda)}{4}$.
		In total we obtain
		\begin{align*}
			&\quad L\Big[
			c_{1}t^{-\frac{3}{2}}e^{-\frac{x}{2}}
			+c_{2}t^{-\lambda}\Big(1 + \gamma \frac{x}{t} \Big)^{-\theta - 2}
			\Big] \\
			&\geq - c(\gamma) c_1 t^{-\frac{7}{2}} \Big(1 + \gamma \frac{x}{t} \Big)^{-\theta -2}
			+ \frac{c_{2}}{2}\Big(2\gamma + \frac{3}{2} - \lambda \Big)t^{-\lambda - 1}\Big(1 + \gamma \frac{x}{t} \Big)^{-\theta -2}  \\
			&\geq \frac{c_{2}}{4}\Big(2\gamma + \frac{3}{2} - \lambda \Big)t^{-\lambda - 1}\Big(1 + \gamma \frac{x}{t} \Big)^{-\theta -2}
		\end{align*}
		whenever $\lambda < \frac{3}{2} + 2 \gamma$ and $c_2$ is large enough in terms of $c_1$ and $\gamma.$ This proves $(a)$. As for $(a'),$ by analogous computation it holds
		\begin{align*}
			&\quad \;L\Big[
			c_{2}t^{-2}\Big(\frac{t}{s} \Big)^{-\frac{3}{2}} \Big(1 + \gamma \frac{x}{t} \Big)^{-5\theta }
			\Big] \\
			&\geq c_{2}t^{-\frac{9}{2}}s^{\frac{3}{2}}
			\Big(1 + \gamma \frac{x}{t} \Big)^{-5\theta} 
			\Big( 
			- \frac{7}{2} 
			+ 5\gamma \theta
			+ 1 - \gamma
			- \frac{5\theta(5\theta + 1)\gamma^2}{t} - \beta (5 \theta - 1) c_j
			\Big)
			\\
			&\geq c_{2}t^{-\frac{9}{2}}s^{\frac{3}{2}}
			\Big(1 + \gamma \frac{x}{t} \Big)^{-5\theta} 
			\Big( 
			4\gamma
			- \frac{(1 + 2\gamma)(1 + 3\gamma)}{\bar{t}_0} - \beta (5 \theta - 1) c_j
			\Big) \\
			&\geq 2 \gamma c_2 t^{-\frac{9}{2}}s^{\frac{3}{2}}
			\Big(1 + \gamma \frac{x}{t} \Big)^{-5\theta} 
		\end{align*}
		whenever $\bar{t}_0 \geq 5\theta(5\theta + 1)\gamma$, so that then
		\begin{align*}
			L\Big[
			\Big(\frac{t}{s} \Big)^{-\frac{3}{2}}e^{-\frac{x}{2}}
			+c_{2}t^{-2}\Big(\frac{t}{s} \Big)^{-\frac{3}{2}}\Big(1 + \gamma \frac{x}{t}\Big)^{-5\theta}
			\Big] 
			&\geq 0
		\end{align*}
		for $c_2$ large enough, completing the proof of $(a')$ and therefore the proof of the theorem under the stronger assumption (\ref{eqn-approxlemma-varphi-bound-exp}).
		
		In the case (\ref{eqn-approxlemma-varphi-bound-f_s}), it remains to compute the left-hand side of $(a)$ and show $(b).$ To do so, we recall (\ref{eqn-def-Tj}), and use \Cref{lemma-hyperbolic-m}, the assumptions of which are satisfied due to $(\ref{approx-lemma-max-prpl-cond})$. 
		We have
		\begin{align*}
			&\quad \; \big| L \big[ S(t,s)\varphi_0 - T(t,s)\varphi_0 \big] \big| \\
			&= \big| L\big[ T(t,s)\varphi_0 \big]\big| 
			= |(m'(t))^3\varphi_0''(m'(t)x + m(t))| 
			\leq  s^{-\frac{7}{2}}(m'(t))^3\Big(1 + \frac{\gamma (m'(t)x + m(t))}{s}\Big)^{-(\theta + 2 )} \\
			&\leq  c(\gamma, \eta) s^{-\frac{7}{2}} \Big(\frac{t}{s}\Big)^{3(\gamma - 1)}\Big(1 + \gamma s^{-1}c_m \Big( x \Big(\frac{t}{s} \Big)^{\gamma - 1} + \gamma^{-1}s\Big( \Big(\frac{t}{s} \Big)^{\gamma } - 1 \Big) \Big) \Big)^{-(\theta + 2)} \\
			&=  c(\gamma, \eta) s^{-\frac{7}{2}} \Big(\frac{t}{s}\Big)^{3(\gamma - 1)}\Big(1 + c_m \Big(\Big(\frac{t}{s} \Big)^{\gamma }\Big(1 + \frac{\gamma x}{t}\Big) - 1 \Big) \Big)^{-(\theta + 2)} \\
			&\leq c(\gamma, \eta) t^{-\frac{7}{2}} \Big(1 + \frac{\gamma x}{t}\Big)^{-\theta - 2}
			\leq \big(2\gamma + \frac{3}{2} - \lambda \big)\frac{c_{2}}{4}t^{-\lambda - 1}\Big(1 + \gamma \frac{x}{t} \Big)^{-\theta -2} 
		\end{align*}
		for $c_2$ sufficiently large, where we have used that $\lambda < \frac{3}{2} + 2 \gamma < \frac{5}{2}$. This yields $(a)$. 
		As for $(b)$, we find that
		\begin{align*}
			|T(t,s)\varphi_0(0)| &= 
			\big|m'(t)\varphi_0(m(t)) \big|
			\leq  m'(t) s^{-\frac{3}{2}}\Big(1 + \frac{\gamma m(t)}{s}\Big)^{-\theta } 
			\leq  c(\gamma, \eta)s^{-\frac{3}{2} - \gamma + 1 + \gamma \theta} t^{\gamma - 1 - \gamma \theta} \\
			&=  c(\gamma, \eta) t^{-\frac{3}{2}}
			\leq c_1t^{-\frac{3}{2}}
		\end{align*}
		for $c_1$ sufficiently large.
		This completes the proof.
	\end{proof}

\subsection{Proof of \Cref{thm-main-parabolic}}
With this, we are in a position to prove \Cref{thm-main-parabolic}.

\begin{proof}
	We replace $g = f - f_s$ and $G = F - F_s$ by
	\begin{align*}
		f_{bl}(x,t) &= \mu e^{-x} + \frac{c_s}{t^{\frac{3}{2}}}e^{-x} = t^{-\frac{3}{2}}F_{bl}(y, \tau), \\
		n_{bl}(t) &= \mu - \frac{c_s}{t^{\frac{3}{2}}},\\
		g(x,t) &= f - f_s - f_{bl}
		= t^{-\frac{3}{2}}G(y, \tau).
	\end{align*}
	since then $g(0, t) = G(0, \tau) = 0$ for all $t \geq t_0$ resp. $\tau \geq \tau_0$. We take $\bar{t}_0$ large enough such that $n_{bl}(t) > 0$ for all $t \geq t_0.$
	Assumption (\ref{thm-main-parabolic-ass-G0}) yields $g(x, t_0) = g_0(x) + \hat{g}_0(x)$ with
	\begin{align}\label{eqn-parabolic-g_0-ass-inproof}
		|g_0(x)| \leq \frac{c_0}{t_0^{\frac{3}{2}}} \Big(1 + \frac{\gamma x}{t_0}\Big)^{-\theta - \varepsilon}, \quad
		|\partial_x^2 g_0(x)| \leq \frac{c_0}{t_0^{\frac{7}{2}}} \Big(1 + \frac{\gamma x}{t_0}\Big)^{-\theta - 2}, \quad
		|\hat{g}_0(x)| \leq (1 + \mu + c_s)t_0^{\frac{1}{2} - \varepsilon \gamma}e^{-\frac{x}{2}}.
	\end{align}
	This implies
	\begin{equation*}
		|n_g(t_0)| \leq \int_{0}^{\infty} x\big(|g_0(x)| + |\hat{g}_0(x)| \big)
		\leq \frac{c_0}{\gamma^2(\theta - 2)( \theta - 1)}  t_0^{\frac{1}{2}}
		+ (1 + \mu + c_s)t_0^{\frac{1}{2} - \varepsilon \gamma}
	\end{equation*}
	and therefore 
	\begin{equation}\label{eqn-parabolic-ng-assumption}
		|n_g(t)| \leq  c_N t_0^{\varepsilon \gamma }t^{\frac{1}{2} - \varepsilon \gamma}, \quad |N_G(\tau)| \leq c_Ne^{- \varepsilon \gamma (\tau- \tau_0)}
	\end{equation}
	on a maximal interval $[t_0, t^*)$ resp. $[\tau_0, \tau^*)$ for
	\begin{equation}\label{eqn-parabolic-cg-lowerbound}
		c_N \geq 2\Big(\frac{c_0}{\gamma^2(\theta - 2)( \theta - 1)} + (1 + \mu + c_s)t_0^{- \varepsilon \gamma} \Big).
	\end{equation}
	We will again show that for $c_0$ small and $t_0$ large enough and a suitable choice of $c_N$, this implies $|N_G(\tau_0)| \leq \frac{c_N}{2} e^{-\varepsilon \gamma (\tau - \tau_0)}$ for all $\tau \in [\tau_0, \tau^*)$, so that $(\ref{eqn-parabolic-ng-assumption})$ can be extended to $[\tau_0, \infty)$.
	
	\textbf{Step 1: } Smallness of $j$

	As in Step 1 of the proof of \Cref{thm-main-hyperbolic}
	we assume $c_0$ is small and $\bar{t}_0$ large enough so that we may additionally assume
	\begin{equation}\label{eqn-parabolic-cg-upperbound}
		c_N \leq \frac{N_s}{4}.
	\end{equation}
	Then it holds that
	$n_s^2 + 2n_g n_s \geq \frac{N_s^2 t}{2} = \frac{n_s^2}{2}$ and $n_s + n_g \geq 0$, so that
	\begin{align*}
		n_s^4\Big(1 + \frac{1 + n_g^2 + n_{bl}^2 + 2n_gn_s + 2n_{bl} (n_g+2n_s) }{n_s^2}\Big)
		\geq n_s^4  \frac{n_s^2 + 2n_gn_s + 2n_{bl} (n_g +n_s)}{n_s^2}
		\geq \frac{N_s^4 t^2}{2}
	\end{align*}
	and therefore, since $\varepsilon \gamma \in \Big(0 , \frac{1}{2}\Big]$ and using (\ref{eqn-parabolic-cg-upperbound}) it holds
	\begin{align*}\label{eqn-parabolic-h-small}
		|h(n_g, t)| 
		&=  \frac{1 + n_g^2 + n_{bl}^2 +  2n_{bl} (n_g+2n_s) 
			- 2n_g n_s^{-1}\big(1 + n_g^2 + n_{bl}^2 + 2n_gn_s + 2n_{bl} (n_g+2n_s)  \big)
		}
		{n_s^4 \Big(1 + \frac{1 + n_g^2 + n_{bl}^2 + 2n_gn_s + 2n_{bl} (n_g+2n_s) }{n_s^2}\Big)} \\
		&\leq \frac{c(\gamma)}{t^2}\Big(t^{\frac{1}{2}} + c_N t_0^\eta t^{\frac{1}{2} - \eta} + c_N^2 t_0^{2\eta} t^{1 - 2 \eta} \Big)
		\leq c(\gamma)(t_0^{- \eta} + c_N^2)t_0^{\eta} t^{-1 -  \eta} .
	\end{align*}
	We remark that the difference to Step 1 of \Cref{thm-main-hyperbolic} lies in the factor $t^{\frac{1}{2}}$ in the first inequality. In the hyperbolic formulation this can be replaced by $1$. Here it is due to the presence of $f_{bl}$ and is the reason why we are restricted now to $\varepsilon \gamma \leq \frac{1}{2}.$
	
	For $j$ it follows
	\begin{equation}\label{eqn-parabolic-j-small}
		| j(t)|
		\leq c(\gamma)
		(t_0^{-\varepsilon \gamma} + c_N + c_N^2)t_0^{\varepsilon \gamma} t^{-1 - \varepsilon \gamma}
		= c_j t_0^{\varepsilon \gamma} t^{-1 - \varepsilon \gamma}.
	\end{equation}
	
	We observe that \Cref{lemma-hyperbolic-m} holds also for $j$ defined by (\ref{eqn-def-j-para}) instead of by (\ref{eqn-def-j}), and the same is true for \Cref{lemma-hyperbolic-M}. We can therefore use them in the following step together with \Cref{approx-lemma} to estimate the terms in (\ref{eqn-ng-duhamel-parabolic}) in the following step.
	
	\textbf{Step 2: } Approximation lemma and hyperbolic estimates
	
	We set $\lambda = \frac{3}{2} + \varepsilon \gamma$ and observe we can take $\bar{\tau}_0$ resp. $\bar{t}_0$ large enough and $c_0$ and therefore $c_N$ small enough such that $c_j$ satisfies $(\ref{approx-lemma-max-prpl-cond})$. So we can apply \Cref{approx-lemma} to approximate terms of the form $S(t,r)\varphi_0$ in $(\ref{eqn-ng-duhamel-parabolic})$. Those terms with exponential decay can be estimated directly, while those with polynomial decay can be estimated by $T(t,r)\varphi_0$. For those terms, as in Step 2 of \Cref{thm-main-hyperbolic}, we apply \Cref{lemma-hyperbolic-m} with $\eta = \varepsilon \gamma \in (0, 1 - \gamma ) \subset (0, 1)$ to estimate $T(t,r)\varphi_0$. 
	
	We begin with $\partial_x (xf_s)$. Clearly this satisfies (\ref{eqn-approxlemma-varphi-bound-f_s}), so by \Cref{lemma-hyperbolic-m} and \Cref{approx-lemma} it holds
	\begin{align*}
		&\quad \int_{t_0}^{t} 
		\beta \Big(\frac{2n_g(r) }{n_s^3(r)}
		+ h(n_g(r), r)\Big) \int_{0}^{\infty} x S(t, r)  \partial_x(xf_s(x,r)) \d x
		\d r \\
		&= \int_{t_0}^{t} 
		\beta \Big(\frac{2n_g(r) }{n_s^3(r)}
		+ h(n_g(r), r)\Big) \Big( \int_{0}^{\infty} x T(t, r)  \partial_x(xf_s(x,r)) \d x + \tilde{\delta}_1(r) \Big)
		\d r
	\end{align*}
	where $|\tilde{\delta}_1(r)|  \leq c(\gamma, \varepsilon)  t^{\frac{1}{2} - \varepsilon \gamma}.$
	
	Both $g_0$ and $r^{\frac{3}{2}}\partial_x^2 f_s(x,r)$ also fulfill (\ref{eqn-approxlemma-varphi-bound-f_s}), so that 
	\begin{align*}
		&\quad \, \Big| \int_{0}^{\infty} x S(t, t_0)g_0 \d x 
		+ \int_{t_0}^{t} \int_{0}^{\infty} x S(t,r) \partial_x^2 f_s(x,r)  \d x \d r
		\Big|
		\\
		&\leq \Big| \int_{0}^{\infty} x T(t, t_0)g_0 \d x  \Big|
		+ c(\gamma, \varepsilon) t^{\frac{1}{2} - \varepsilon \gamma}
		+ \int_{t_0}^{t}  r^{-\frac{3}{2}}\Big( \Big| \int_{0}^{\infty} x T(t,r) r^{\frac{3}{2}}\partial_x^2 f_s(x,r) \d x \Big| + c(\gamma, \varepsilon) \Big)  \d r.
	\end{align*}
	Furthermore, using \Cref{lemma-hyperbolic-m}$(i)-(ii)$, very similarly to the proof of \Cref{lemma-hyperbolic-m}$(iii)$, one can show that for any $|\varphi_0(x)| \leq t_0^{-\frac{3}{2}}\Big( 1 + \frac{\gamma x }{t_0} \Big)^{-\theta - \varepsilon}$ it holds
	\begin{equation*}
		\Big|\int_{0}^{\infty} x T(t, t_0)\varphi_0 \d x \Big| \leq c(\gamma)t^{\frac{1}{2}} \Big(\frac{t}{t_0} \Big)^{- \varepsilon \gamma}.
	\end{equation*}
	By assumption $g_0$ and $r^{\frac{3}{2}}\partial_x^2 f_s(x,r)$ fulfill this condition (the latter even with $\varepsilon = 2$), so that
	\begin{align*}
		\Big| \int_{0}^{\infty} x T(t, t_0)g_0 \d x  \Big|
		\leq c_0 c(\gamma)t^{\frac{1}{2}} \Big(\frac{t}{t_0} \Big)^{-\varepsilon \gamma}
	\end{align*}
	and 
	\begin{equation*}
		\Big| \int_{t_0}^{t} \int_{0}^{\infty} x T(t,r) \partial_x^2 f_s(x,r) \d x \d r \Big|
		\leq c(\gamma, \varepsilon)\int_{t_0}^{t} t^{\frac{1}{2} - \varepsilon \gamma} r^{-\frac{3}{2} + \varepsilon \gamma} \d r
		\leq c(\gamma, \varepsilon).
	\end{equation*}
	In total we obtain
	\begin{align*}
		\Big| \int_{0}^{\infty} x S(t, t_0)g_0 \d x 
		+ \int_{t_0}^{t} \int_{0}^{\infty} x S(t,r) \partial_x^2 f_s(x,r)  \d x \d r
		\Big|
		&\leq c(\gamma, \varepsilon)\big(c_0 + t_0^{-\varepsilon \gamma} \big)t^{\frac{1}{2}}\Big(\frac{t}{t_0} \Big)^{- \varepsilon \gamma}.
	\end{align*}
	Furthermore, $\hat{g}_0$ satisfies (\ref{eqn-approxlemma-varphi-bound-exp}), so that by \Cref{approx-lemma}
	\begin{align*}
		\Big| \int_{0}^{\infty} x S(t, t_0)\hat{g}_0 \d x \Big| 
		&\leq c(\gamma)(1 + \mu + c_s)t_0^{\frac{1}{2} - \varepsilon \gamma}\Big(\frac{t}{t_0}\Big)^{-\frac{3}{2}}
		\leq c(\gamma)(1 + \mu + c_s)t_0^{- \varepsilon \gamma}t^{\frac{1}{2}} \Big(\frac{t}{t_0} \Big)^{- \varepsilon \gamma} \\
		&\leq c(\gamma)t_0^{- \varepsilon \gamma}t^{\frac{1}{2}} \Big(\frac{t}{t_0} \Big)^{- \varepsilon \gamma}.
	\end{align*}
	Next, we consider $\mathcal{R}_{bl}$, the contributions coming from the boundary layer approximation. From (\ref{eqn-def-Rbl}) and (\ref{eqn-parabolic-j-small}) we obtain that
	$|\mathcal{R}_{bl}(x,r)| \leq c(\gamma)r^{-1} e^{-\frac{x}{2}},$
	so that by \Cref{approx-lemma} and $\varepsilon \gamma \leq \frac{1}{2}$ it holds
	\begin{align*}
				\Big| \int_{t_0}^{t} \int_{0}^{\infty} x S(t, r)\mathcal{R}_{bl}(x,r) \d x \d r \Big|
				&\leq c(\gamma, \varepsilon) \int_{t_0}^{t} \Big(\frac{t}{r} \Big)^{-\frac{3}{2}}r^{-1} \d r  \leq c(\gamma, \varepsilon) \\
				&\leq c(\gamma, \varepsilon)\Big(\frac{t}{t_0}\Big)^{\frac{1}{2} - \varepsilon \gamma}.
	\end{align*}
	
	All in all, we have
	\begin{align*}
		\delta_2(r) &\coloneqq 
		\Big| \int_{0}^{\infty} x S(t, t_0)(g_0 + \hat{g}_0) \d x 
		+ \int_{t_0}^{t} \int_{0}^{\infty} x S(t,r)\big(\mathcal{R}_{bl}(x,r) + \partial_x^2 f_s(x,r) \big) \d x \d r
		\Big| \\
		&\leq c(\gamma, \varepsilon)\big(
		c_0 + t_0^{-\varepsilon \gamma}   \big)
		t^{\frac{1}{2}} \Big(\frac{t}{t_0} \Big)^{- \varepsilon \gamma}
	\end{align*}
	and rewrite (\ref{eqn-ng-duhamel-parabolic}) to
	\begin{align*}
		n_g(t) = \int_{t_0}^{t} 
		\beta \Big(\frac{2n_g(r) }{n_s^3(r)}
		+ h(n_g(r), r)\Big) \Big( \int_{0}^{\infty} x T(t, r)  \partial_x(xf_s(x,r)) \d x + \tilde{\delta}_1(r) \Big)
		\d r + \delta_2(t).
	\end{align*}
	We are now ready to move back to self-similar variables using the relations described in \Cref{prop-hyperbolic_NG}, which also imply that
	\begin{equation*}
		\int_{0}^{\infty} xT(t,t_0)\varphi \d x
		= t^{-\frac{3}{2}}\int_{0}^{\infty} x\tilde{T}(\tau,\tau_0)\Phi \d x
		= t^{\frac{1}{2}}\int_{0}^{\infty} y\tilde{T}(\tau,\tau_0)\Phi \d y.
	\end{equation*}
	Let $H$ and $\tilde{T}$ be as in \Cref{prop-hyperbolic_NG}, $\tilde{\Delta}_1(\tau) = t^{-\frac{1}{2}} \tilde{\delta}_1(t)$ and $\Delta_2(\tau) = t^{-\frac{1}{2}}\delta_2(t).$ Then
	\begin{align*}
		N_G(\tau) &= t^{-\frac{1}{2}} n_g(t) \\
		&=  \int_{\tau_0}^{\tau} 
		\beta \Big(\frac{2N_G(r) }{N_s^3}
		+ H(N_G(r), r)\Big) 
		\Big(
		\int_{0}^{\infty} y \tilde{T}(\tau, r)  \partial_y(yF_s(r)) \d x
		+ \tilde{\Delta}_1(r)
		\Big)
		\d r
		+ \Delta_2(\tau),
	\end{align*}
	which as in Step 2 of \Cref{thm-main-hyperbolic} can be further rewritten using \Cref{lemma-hyperbolic-M}(iii) to
	\begin{align*}
		N_G(\tau)
		&= - \frac{1 - \gamma}{\gamma} \int_{\tau_0}^{\tau} 
		\Big( 
		N_G(r)+ \frac{N_s^3}{2\beta}H(r)
		\Big) 
		\Big(
		1 
		- (1 - 2\gamma) \Big(\frac{t}{r} \Big)^{-\gamma}
		+ \Delta_1(r)
		\Big)
		\d r
		+ \Delta_2(\tau),
	\end{align*}
	where we now take $\Delta_1$ comprised of the parabolic-to-hyperbolic approximation error $\tilde{\Delta}_1$ and the error $\Delta_1$ from \Cref{lemma-hyperbolic-M}(iii), which is already present in the hyperbolic case. We then have
	\begin{align*}
		|\Delta_1(r)| &\leq c(\gamma, \varepsilon)\big( c_j + e^{-\varepsilon \gamma \tau} \big)
		\leq c(\gamma, \varepsilon)\big( c_N + c_N^2 + e^{-\varepsilon \gamma \tau_0} \big),
		\\
		|\Delta_2(\tau) | &\leq c(\gamma, \varepsilon)\big(c_0 + e^{-\varepsilon \gamma \tau_0} \big)
		e^{-\varepsilon \gamma (\tau - \tau_0)}, \\
		\big|\frac{N_s^3}{2\beta}H(r)\big| &\leq c(\gamma)(e^{-\varepsilon \gamma \tau_0} + c_N^2)e^{-\varepsilon \gamma (\tau - \tau_0)}.
	\end{align*}
	
	\textbf{Step 3: } Conclusion of continuation argument
	
	We follow Step 3 in the proof of \Cref{thm-main-hyperbolic}: When $\tau_0$ is sufficiently large and $c_0$ sufficiently small, we may take $c_N$ small enough that the bound (\ref{lemma-ODE-in-selfsim_c1}) on $\Delta_1$ holds without violating (\ref{eqn-parabolic-cg-lowerbound}), so that we can apply \Cref{lemma-ODE-in-selfsim}. Doing so yields
	\begin{align}\label{eqn-parabolic_NG-conclusion}
		| N_G(\tau) |
		&\leq c(\gamma, \varepsilon)\big(e^{- \varepsilon \gamma \tau_0} + c_N^2 +  c_0  \big)e^{-  \varepsilon \gamma (\tau - \tau_0)}.
	\end{align}
	We assume $c_0$ is small enough such that
	$c(\gamma, \varepsilon)c_N^2 \leq \frac{c_N}{4}$
	without violating (\ref{eqn-parabolic-cg-lowerbound}), and that $\tau_0$ is large enough such that
	\begin{equation}\label{eqn-parabolic-cg-lowerbound2}
		c(\gamma,\varepsilon) \big( c_0 + e^{-\varepsilon \gamma \tau_0} \big) \leq \frac{c_N}{4}
	\end{equation}
	without violating (\ref{eqn-parabolic-cg-upperbound}) (note this is, up to constants, the same upper lower bound as (\ref{eqn-parabolic-cg-lowerbound})).
	Then (\ref{eqn-parabolic_NG-conclusion}) reads
	$|N_G(\tau)| \leq \frac{c_N}{2} e^{-  \varepsilon \gamma (\tau - \tau_0)},$
	so that the assumption (\ref{eqn-parabolic-ng-assumption}) can be extended from $[t_0, t^*)$ to $[t_0, \infty)$ and from $[\tau_0, \tau^*)$ to $[\tau_0, \infty)$, respectively. The optimal choice of $c_N$ is determined by the lower bounds (\ref{eqn-parabolic-cg-lowerbound}) and (\ref{eqn-parabolic-cg-lowerbound2}).
	
	We thus have $|N_G(\tau)| = |N_F(\tau) - N_s - N_{bl}(\tau)| \leq c_Ne^{-  \varepsilon \gamma (\tau - \tau_0)}$ for all $\tau \geq \tau_0$. 
	
	\textbf{Step 4: } Limit of $F$

	We now turn our eye to the limit of $f - f_s - \mu e^{-x},$ which is slightly more involed than in the proof of Theorem \ref{thm-main-hyperbolic}. We again use \Cref{approx-lemma} with $\lambda = \frac{3}{2} +  \gamma$. This time, we want to apply it to a modified version of $f$ having boundary value $0$. In fact we apply it to
	$$f - \mu e^{-x} - \frac{\beta \mu}{2N_s^2} \frac{x^2}{t}e^{-x},$$ 
	since this yields a Duhamel formula for $f$ with better decay. In fact $\frac{\beta \mu}{2N_s^2} \frac{x^2}{t}e^{-x}$ is the next term that would appear in an asymptotic expansion for $f$ at the boundary. It holds
	\begin{align*}
		L\Big[
		\Big(\mu e^{-x} + \frac{\beta \mu}{2N_s^2} \frac{x^2}{t}e^{-x} \Big) \Big] 
		&= j(t) \partial_x \Big(x \Big(\mu e^{-x} - \frac{\beta \mu}{N_s^2} \frac{x^2}{2t}e^{-x} \Big)\Big) 
		-\frac{\beta \mu}{2N_s^2t^{2}}t^{-2}x^2e^{-x} + \frac{\beta^2 \mu}{2N_s^4 t^2} \partial_x(x^3e^{-x}) \\
		&\eqqcolon \mathcal{R}(x,t)
	\end{align*}
	for $|\mathcal{R}(x,t)| \leq c(\gamma)t_0^{\varepsilon \gamma}t^{-1-\varepsilon \gamma}e^{-\frac{x}{2}}$ using (\ref{eqn-parabolic-j-small}). It will be key that this decays more quickly than $L[\mu e^{-x}]$, which we recall satisfies $|L[\mu e^{-x}]| \leq c(\gamma)t^{-1}$.
	
	By Duhamel's formula it holds
	\begin{align*}\label{eqn-duhamel-f}
		&\quad \; f(x,t) -\mu e^{-x} - \frac{\beta \mu}{N_s^2} \frac{x^2}{2t}e^{-x} \nonumber \\
		&= S(t,t_0)\Big(f_0(x) - \Big(\mu e^{-x} + \frac{\beta \mu}{N_s^2} \frac{x^2}{2t_0}e^{-x}\Big)\Big) + \int_{t_0}^{t} S(t,r)\mathcal{R}(x,r) \d r \nonumber \\
		&= S(t,t_0)\Big(g_0 + \hat{g}_0 + f_s(t_0) + \frac{c_s}{t^{\frac{3}{2}}}e^{-x} - \frac{\beta \mu}{2N_s^2} \frac{x^2}{t}e^{-x}\Big)
		+ \int_{t_0}^{t} S(t,r)\mathcal{R}(x,r) \d r.
	\end{align*}
	By (\ref{eqn-parabolic-g_0-ass-inproof}), $g_0 + f_s(t_0)$ satisfy (\ref{eqn-approxlemma-varphi-bound-f_s}), while 
	$$\Big| \hat{g}_0 + \frac{c_s}{t^{\frac{3}{2}}}e^{-x} - \frac{\beta \mu}{2N_s^2} \frac{x^2}{t}e^{-x} \Big| 
	\leq c(\gamma)t_0^{\frac{1}{2} - \varepsilon \gamma}e^{-\frac{x}{2}}
	$$
	satisfies (\ref{eqn-approxlemma-varphi-bound-exp}), so that by \Cref{approx-lemma}
	\begin{align*}
		&\quad \Big| S(t,t_0)\Big(g_0 + \hat{g}_0 + f_s(t_0) + \frac{c_s}{t^{\frac{3}{2}}}e^{-x} - \frac{\beta \mu}{2N_s^2} \frac{x^2}{t}e^{-x}\Big) 
		- T(t,t_0) (g_0 + f_s(t_0))
		\Big| \\
		&\leq c(\gamma)
		\Big(t^{-\frac{3}{2}}\big(1 + t_0^{2 - \varepsilon \gamma}\big)e^{-\frac{x}{2}} 
		+ t^{-\frac{3}{2}}\big(t^{ - \gamma} + t^{-2}t_0^{2 - \varepsilon \gamma} \big)
		\Big(1 + \frac{\gamma x}{t}\Big)^{-\theta - 2}
		\Big) \\
		&\leq c(\gamma, t_0)t^{-\frac{3}{2}} \Big(e^{-\frac{x}{2}} + t^{ - \gamma}\Big(1 + \frac{\gamma x}{t}\Big)^{-\theta - 2} \Big).
	\end{align*}
	As for $\mathcal{R}(x,t)$, (\ref{eqn-approxlemma-varphi-bound-exp}) is satisfied, so by \Cref{approx-lemma} we obtain 
	\begin{align*}
		\Big| \int_{t_0}^{t} S(t,r)\mathcal{R}(x,r) \d r \Big|
		&\leq c(\gamma)t_0^{\varepsilon \gamma} \int_{t_0}^{t} \Big(
		\Big(\frac{t}{r}\Big)^{-\frac{3}{2}} e^{-\frac{x}{2}}
		+ t^{-2}\Big(\frac{t}{r}\Big)^{-\frac{3}{2}} \Big(1 + \frac{\gamma x}{t}\Big)^{-5\theta }
		\Big)
		r^{-1 - \varepsilon \gamma} \d r \\
		&\leq c(\gamma)\Big(\frac{t}{t_0}\Big)^{-\varepsilon \gamma}
		\Big(e^{-\frac{x}{2}} + t^{-2}\Big(1 + \frac{\gamma x}{t}\Big)^{-5\theta } \Big).
	\end{align*}
	Finally, it holds
	$$\Big| \frac{\beta \mu}{2N_s^2} \frac{x^2}{t}e^{-x} \Big|
	\leq c(\gamma)t^{-1}e^{-\frac{x}{2}},
	$$
	so that we obtain in total
	\begin{align*}
		\big|f(x,t) -\mu e^{-x} -  T(t,t_0)(g_0 + f_s(t_0)) \big|
		&\leq c(\gamma, t_0) \Big( t^{-\varepsilon \gamma} e^{-\frac{x}{2}}
		+ t^{-\frac{3}{2} - \gamma}\Big(1 + \frac{\gamma x}{t}\Big)^{-\theta - 2} \Big).
	\end{align*} 
	With this, we are ready to pass to self-similar variables, in which the above reads
	\begin{align*}
		\big|F - \mu e^{\frac{3}{2}\tau - ye^{\tau}} - \tilde{T}(\tau, \tau_0)(G_0 + F_s) \big|
		&\leq c(\gamma, \tau_0)
		\Big(
		e^{(\frac{3}{2} - \varepsilon \gamma )\tau - \frac{ye^{\tau}}{2}} + e^{-\gamma \tau} \Big(1 + \gamma y \Big)^{-\theta - 2}
		\Big).
	\end{align*}
	Next, we notice that
	\Cref{lemma-F-to-Fs} holds also for $J(\tau) = tj(t)$ where $j$ is defined by (\ref{eqn-def-j}), since we still have $|J(\tau)| \leq t|j(t)| \leq c_j \Big(\frac{t}{t_0}\Big)^{-\varepsilon \gamma} = e^{-\varepsilon \gamma (\tau - \tau_0)}.$ Furthermore, by assumption (\ref{eqn-parabolic-g_0-ass-inproof}) it holds 
	$$\lim\limits_{y \to \infty} y^{-\theta}(G_0 + F_s) = c_s\gamma^{-\theta},
	$$
	so by \Cref{lemma-F-to-Fs} it holds
	$$\tilde{T}(\tau, \tau_0)(G_0 + F_s) \to F_s(y).$$
	as $\tau \to \infty$ for all $y \geq 0$.
	Thus we have in total
	\begin{align*}
		&\quad \lim\limits_{\tau \to \infty} 
		\Big|
		\frac{F - F_s - \mu e^{\frac{3}{2}\tau - ye^{\tau}}}{ e^{\frac{3}{2}\tau - ye^{\tau}} + (1 + \gamma y)^{-\theta}}
		\Big| \\
		&\leq \lim\limits_{\tau \to \infty}
		\big|\tilde{T}(\tau, \tau_0)(G_0 + F_s) - F_s \big|
		+ c(\gamma, \tau_0)\lim\limits_{\tau \to \infty}
		\frac{
			e^{(\frac{3}{2} - \varepsilon \gamma )\tau - \frac{ye^{\tau}}{2}} + e^{-\gamma \tau} (1 + \gamma y )^{-\theta - 2}
			}
			{ e^{\frac{3}{2}\tau - ye^{\tau}} + (1 + \gamma y)^{-\theta}}
		= 0
	\end{align*}
	for all $y \geq 0$, and uniformly on compact sets in $[0, \infty)$ by Dini's theorem.   
\end{proof}

\bibliographystyle{alpha}
\bibliography{Electroporation}
\end{document}